\documentclass[12pt]{amsart}
\usepackage[cp850]{inputenc}
\usepackage{graphics}
\usepackage{enumitem}

\newtheorem{thm}{Theorem}[section]
\newtheorem*{thmA}{Theorem A}
\newtheorem*{thmB}{Theorem B}
\newtheorem*{thmA'}{Theorem A'}
\newtheorem*{thmB'}{Theorem B'}
\newtheorem*{thmAbis}{Theorem \^{A}}
\newtheorem*{thmA''}{Theorem A''}
\newtheorem*{thmB''}{Theorem B''}

\newtheorem{pro}[thm]{Proposition}
\newtheorem{cor}[thm]{Corollary}

\theoremstyle{definition}
\newtheorem{defi}[thm]{Definition}

\newtheorem{rmk}[thm]{Remark}

\numberwithin{equation}{section}

\usepackage{amssymb,amsmath}

\def\div{\mathop\mathrm{div}\nolimits}

\def\tr{\mathop\mathrm{tr}\nolimits}

\def\Ric{\mathop\mathrm{Ric}\nolimits}
\def\cut{\mathop\mathrm{cut}\nolimits}
\def\Krad{\mathop{K_{\mathrm{rad}}}\nolimits}

\newcommand{\Hess}{\operatorname{Hess}}

\newcommand{\e}{\mbox{$\mathrm{e}$}}
\newcommand{\R}{\mbox{${\mathbb R}$}}

\newcommand{\g}[2]{\mbox{$\langle #1 ,#2 \rangle$}}
\newcommand{\fle}{\mbox{$\rightarrow$}}
\newcommand{\ep}{\varepsilon}
\newcommand{\rf}[1]{\mbox{(\ref{#1})}}
\newcommand{\rl}[1]{{~\ref{#1}}}

\newcommand{\M}{(M,\g{\ }{\ })}
\newcommand{\uin}{u\in\mathcal{C}^2(M)}

\def\beq{\begin{equation}}
\def\eeq{\end{equation}}

\begin{document}

\title[A general form of the weak maximum principle]
{A general form of the weak maximum principle and some applications}

\author{Guglielmo Albanese}
\address{Dipartimento di Matematica, Universit\`{a} degli Studi di Milano, Via Saldini 50, I-20133, Milano, Italy}
\email{guglielmo.albanese@studenti.unimi.it}

\author{Luis J. Al\'\i as}
\address{Departamento de Matem\'{a}ticas, Universidad de Murcia, E-30100 Espinardo, Murcia, Spain}
\email{ljalias@um.es}
\thanks{This work was partially supported by MICINN/FEDER project MTM2009-10418 and Fundaci\'{o}n S\'{e}neca
project 04540/GERM/06, Spain. This research is a result of the
activity developed within the framework of the Programme in Support
of Excellence Groups of the Regi\'{o}n de Murcia, Spain, by
Fundaci\'{o}n S\'{e}neca, Regional Agency for Science and Technology
(Regional Plan for Science and Technology 2007-2010).}

\author{Marco Rigoli}
\address{Dipartimento di Matematica, Universit\`{a} degli Studi di Milano, Via Saldini 50, I-20133, Milano, Italy}
\email{marco.rigoli@unimi.it}
\thanks{M. Rigoli was partially supported by MEC Grant SAB2010-0073}


\date{January 2012; revised September 2012}



\begin{abstract}
The aim of this paper is to introduce new forms of the weak and Omori-Yau maximum principles for linear operators, notably
for trace type operators, and show their usefulness, for instance, in the context of PDE's and in the theory of hypersurfaces.
In the final part of the paper we consider a large class of non-linear operators and we show that our previous results
can be appropriately generalized to this case.
\end{abstract}

\maketitle

\section{Introduction}
\label{s1} A well known result due to Omori \cite{Om} and Yau \cite{Ya,CY}, from now on the Omori-Yau maximum principle,
states that on a complete Riemannian manifold $\M$ with Ricci tensor bounded from below, for any function $\uin$ with
$u^*=\sup_Mu<+\infty$ there exists a sequence $\{x_k\}\subset M$ with the following properties
\beq
\label{1.1}
\text{a) } \, u(x_k)>u^*-\frac{1}{k}, \quad \text{ b) } \, \Delta u(x_k)<\frac{1}{k}, \quad \text{ and c) } \,
|\nabla u|(x_k)<\frac{1}{k}
\eeq
for eack $k\in\mathbb{N}$.

In 2002, Pigola, Rigoli and Setti \cite{PRS1} introduced what has been called the weak maximum principle with the
following definition: we say that the weak maximum principle holds on a Riemannian manifold $\M$ if for any function
$\uin$ with $u^*=\sup_Mu<+\infty$ there exists a sequence $\{x_k\}\subset M$ with the properties a) and b) in \rf{1.1}.

This seemingly simple minded definition is in fact deep: it turns out to be equivalent to the stochastic completeness of
the Riemannian manifold $\M$ as shown in \cite{PRS1}. This latter concept does not require the manifold to be
complete from the Riemannian point of view and a simple useful condition to guarantee stochastic completeness is given by
the Khas'minski{\u\i} test \cite{Ka}, that is, by the existence of a function $\gamma\in\mathcal{C}^2(M)$ such that
\beq
\label{1.2}
\begin{cases}
\text{i)} \quad \gamma(x)\rightarrow +\infty & \hbox{as $x\rightarrow\infty$,} \\
\text{ii)} \quad \Delta\gamma\leq\lambda\gamma & \hbox{outside a compact subset of $M$}
\end{cases}
\eeq
for some positive constant $\lambda>0$.

Thus, we do not necessarily require any curvature conditions to guarantee the applicability of the
principle. This observation applies to the Omori-Yau maximum principle too, as shown in Theorem 1.9 of \cite{PRS2}.
We remark that, very recently, the sufficient condition for stochastic completeness given by the Khas'minski{\u\i} test
has been shown to be in fact also necessary \cite{MV}.

This approach, based on the existence of some auxiliary function satisfying appropriate conditions, has revealed to be of
great versatility in geometric applications; for instance, in the geometry of submanifolds
\cite{AD1,ABD,AD2,AIR1,AIR2} and in the study of soliton structures \cite{FG,MRR,PRRS}.

The purpose of this paper is to prove a weak maximum principle (Theorem A), an Omori-Yau type maximum principle
(Theorem B) and further related results for a large class of linear differential operators of geometrical interest.

From now on $\M$ will denote a connected, Riemannian manifold of dimension $m\geq 2$.
To describe our first result let $T$ be a symmetric positive semi-definite $(2,0)$-tensor field on $M$ and $X$ a vector field. We set
$L=L_{T,X}$ to denote the differential operator acting on $\uin$ by
\beq
Lu=\div (T(\nabla u,\ )^\sharp)-\g{X}{\nabla u}
=\tr(T\circ\Hess(u))+\div T(\nabla u)-\g{X}{\nabla u}
\eeq
where $\ ^{\sharp}:T^{*}M\rightarrow TM$ is the musical isomorphism. For instance if $T=\g{\ }{\ }$ and $X$ is a vector field
on $M$ for $\uin$ we have
\beq
Lu=\Delta u-\g{X}{\nabla u}
\eeq
and $L$ coincides with the $X-$Laplacian, denoted by $\Delta_X$, used in the study of general soliton structures,
\cite{MRR}; in particular if $X=\nabla f$ then $L=\Delta_f$ is the $f-$Laplacian, appearing also as the natural
symmetric diffusion operator in the study of the weighted Riemannian manifold $(M,\g{}{},\e^{-f}d vol)$, \cite{GY}.
On the other hand, if $T$ is as above and $X=(\div T)^{\sharp}$, then for $\uin$, $Lu$ reduces to
\beq
Lu=\tr(T\circ\Hess(u))
\eeq
and it is a typical trace operator.
\begin{thmA}
\label{thmA}
Let $\M$ be a Riemannian manifold and $L=L_{T,X}$ as above. Let $q(x)\in\mathcal{C}^0(M)$, $q(x)\geq0$ and suppose that
\beq
\label{e1.6}
q(x)>0 \ \hbox{outside a compact set.}
\eeq
Let $\gamma\in\mathcal{C}^2(M)$ be such that
\beq\label{Gam}
\tag{$\Gamma$}
\begin{cases}
\text{i)} \quad \gamma(x)\rightarrow +\infty & \hbox{as $x\rightarrow\infty$,} \\
\text{ii)} \quad q(x)L\gamma(x)\leq B & \hbox{outside a compact set}
\end{cases}
\eeq
for some constant $B>0$. If $\uin$ and $u^*<+\infty$, then there exists a sequence $\left\{x_k\right\}\subset M$ with the properties
\beq\label{woy}
\text{a) } \, u(x_k)>u^*-\frac{1}{k}, \quad  \text{ and b) } \, q(x_k)Lu(x_k)<\frac{1}{k}
\eeq
for each  $k\in\mathbb{N}$.
\end{thmA}
If the conclusion of the theorem holds on $\M$ we shall say that the $q-$weak maximum principle for the operator $L$ holds on $\M$. If $q\equiv 1$
we shall say that the weak maximum principle for the operator $L$ holds on $\M$. Obviously, if the $q$-weak maximum principle
holds for $L$ and $0\leq \hat{q}(x)\leq q(x)$, $\hat{q}(x)$ satisfying \rf{e1.6}, then the $\hat{q}$-weak maximum
principle for the operator $L$ also holds.

Note that, if $T=p(x)\g{\ }{\ }$ for some $p\in\mathcal{C}^1(M)$, $p>0$ on $M$, and $X\equiv 0$, then $q(x)L$ is
(at least on the set where $q$ is positive) a typical (non symmetric) diffusion operator.

We stress that the Riemannian manifold $M$ is not assumed to be (geodesically) complete. This matches with the fact that for
$L=\Delta$ and $q(x)\equiv 1$, condition \rf{Gam} i), ii) (see also Remark \rf{rmkA}) is exactly the Khas'minski{\u\i}
condition that we have mentioned above.

\begin{rmk}\label{rmkA}
As we shall show below, condition ii) in (\ref{Gam}) can be substituted, for instance, by
\beq\label{Gam'}
\tag{$\Gamma$}
\begin{array}{ccc}
\text{ii)}' \quad q(x)L\gamma(x)\leq G(\gamma(x)) & {} & \hbox{outside a compact subset of $M$}
\end{array}
\eeq
where $G\in\mathcal{C}^1(\mathbb{R}^+)$ is non negative and satisfies
\beq\label{G1}
\begin{array}{ccc}
\text{ i) }\quad\frac{1}{G}\notin L^1(+\infty); & {} & \text{ii)}\quad G^{\prime}(t)\geq-A(\log t +1),
\end{array}
\eeq
for $t>>1$ and some constant $A\geq 0$. For instance, the functions $G(t)=t$, $G(t)=t\log t$, $t>>1$, $G(t)=t\log t\log\log t$, $t>>1$, and so on,
satisfy i) and (ii) in (\ref{G1}) with $A=0$.

It seems worth to underline the following fact. In \cite{PRS1} the third author, jointly with Pigola and Setti, proved
that the weak maximum principle for $\Delta$ is equivalent to the stochastic completeness of the manifold $M$ via the known characterization
(see Grigor'yan \cite{Gr} or \cite{PRS2}) that $(M,\g{}{})$ is stochastically complete if and only if for each
$\lambda>0$ the only non-negative bounded solution of $\Delta u=\lambda u$ is $u\equiv 0$. The work of Mari and Valtorta \cite{MV} shows
that the weak maximum principle implies the existence of a function $\gamma$ satisfying Khas'minski{\u\i} criterion \rf{1.2}. This latter classically
implies stochastic completeness (see \cite{PRS2} for a simple proof using the equivalence mentioned above). Theorem A
above provides a direct proof of the weak maximum principle starting from Kash'minski test.
\end{rmk}

The ''Omori-Yau'' type version of Theorem A is as follows.
\begin{thmB}
Let $\M$ be a Riemannian manifold and $L$ as above. Let $q(x)\in\mathcal{C}^0(M)$, $q(x)\geq 0$ and suppose
\beq
\begin{array}{ccc}
q(x)>0& {} & \hbox{outside a compact set}.
\end{array}
\eeq
Let $\gamma\in\mathcal{C}^2(M)$ be such that
\beq
\label{GamB}
\tag{$\Gamma_B$}
\begin{cases}
\text{i)} \quad \gamma(x)\rightarrow +\infty & \hbox{as $x\rightarrow\infty$,} \\
\text{ii)} \quad q(x)L\gamma\leq B & \hbox{outside a compact subset of $M$,} \\
\text{iii)} \quad |\nabla\gamma|\leq A & \hbox{outside a compact subset of $M$}
\end{cases}
\eeq
for some constants $A,B>0$. If $\uin$ and $u^*<+\infty$ then there exists a sequence $\left\{x_k\right\}\subset M$ with the properties
\beq\label{oy}
\text{a) } \, u(x_k)>u^*-\frac{1}{k}, \quad \text{ b) } \, q(x_k)Lu(x_k)<\frac{1}{k}, \quad \text{ and c) } \,
|\nabla u(x_k)|<\frac{1}{k}
\eeq
for each  $k\in\mathbb{N}$.
\end{thmB}

\begin{rmk}\label{rmkB}
In this case conditions ii) and iii) in (\ref{GamB}) can be substituted by the apparently weaker request
\beq
\label{GamB'}
\tag{$\Gamma_B$}
\begin{cases}
\text{ii)}' \quad q(x)L\gamma\leq G(\gamma) \\
\text{iii)}' \quad |\nabla\gamma|\leq G(\gamma)
\end{cases}
\eeq
outside a compact subset of $M$, where $G\in\mathcal{C}^1(R_0^+)$ is a positive function satisfying $(\ref{G1})$ $i)$, $ii)$.
\end{rmk}

We observe that when $\M$ is a complete, non-compact Riemannian manifold a special candidate for $\gamma$, in both Theorems A and B,
is the distance function $r(x)$ from a fixed
origin $o\in M$. Of course $r(x)$ is smooth only outside $\left\{o\right\}\cup\cut(o)$, where $\cut(o)$ is the cut locus
of $o$, but, as we shall show at the end of the proof of Theorem B, this problem can be bypassed using an old trick of
Calabi \cite{Ca}. Needless to say, the inequalities involving $r(x)$ and the operator $L$ have to be understood in the
weak-Lip sense. We underline that the arguments we shall give below, via a comparison principle, also shows that if
$\gamma\in\mathcal{C}^1(M)$ satisfies (\ref{GamB}) i), iii), and is a classical weak solution of (\ref{GamB}) ii),
then Theorem B is still valid. The same, of course, applies to Theorem A and to the regularity of $u$ (but in this
latter case with the further assumption $1/q\in L^1_{loc}(M)$ and the application of Theorem 5.6 of \cite{PuRS}
when proving that $u^*$ is not attained on $M$; see the proof of Theorem A'').

On the other hand, given $T$ and $X$ as above we introduce the operator $H=H_{T,X}$ acting on $\mathcal{C}^2(M)$ by
\[
Hu=H_{T,X}u=T(\mathrm{hess}(u)\cdot,\cdot)+(\mathrm{div}T-X^\flat)\otimes du.
\]
Observe that $Lu=\mathrm{tr}(Hu)$. Then, the above Theorems admit the following general versions.
\begin{thmA'}
\label{thmA'}
Let $\M$ be a Riemannian manifold and $H=H_{T,X}$ be as above. Let $q(x)\in\mathcal{C}^0(M)$, $q(x)\geq0$ and suppose that
\beq
q(x)>0 \ \hbox{outside a compact set.}
\eeq
Let $\gamma\in\mathcal{C}^2(M)$ be such that
\beq\label{GamC}
\tag{$\Gamma_C$}
\begin{cases}
\text{i)} \quad \gamma(x)\rightarrow +\infty \quad \hbox{ as $x\rightarrow\infty$,} \\
\text{ii)} \quad q(x)H\gamma(x)(v,v)\leq B|v|^2
\end{cases}
\eeq
for some constant $B>0$ and for every $x\in M\setminus K$, for some compact $K\subset M$, and for every $v\in T_xM$.
If $\uin$ and $u^*<+\infty$, then there exists a sequence $\left\{x_k\right\}\subset M$ with the properties
\beq\label{woy2}
\text{i) } \, u(x_k)>u^*-\frac{1}{k}, \quad  \text{ and ii) } \, q(x_k)Hu(x_k)(v,v)<\frac{1}{k}|v|^2
\eeq
for each  $k\in\mathbb{N}$ and every $v\in T_{x_k}M, v\neq 0$.
\end{thmA'}

\begin{thmB'}
\label{thmB'}
Let $\M$ be a Riemannian manifold and $H=H_{T,X}$ be as above. Let $q(x)\in\mathcal{C}^0(M)$, $q(x)\geq0$ and suppose that
\beq
q(x)>0 \ \hbox{outside a compact set.}
\eeq
Let $\gamma\in\mathcal{C}^2(M)$ be such that
\beq\label{GamD}
\tag{$\Gamma_D$}
\begin{cases}
\text{i)} \quad \gamma(x)\rightarrow +\infty \quad \hbox{as $x\rightarrow\infty$,} \\
\text{ii)} \quad q(x)H\gamma(x)(v,v)\leq B|v|^2, \\
\text{iii)} \quad |\nabla\gamma(x)|\leq A \\
\end{cases}
\eeq
for some constants $A, B>0$, for every $x\in M\setminus K$, for some compact $K\subset M$, and for every $v\in T_xM$.
If $\uin$ and $u^*<+\infty$, then there exists a sequence $\left\{x_k\right\}\subset M$ with the properties
\beq\label{woy3}
\text{i) } \, u(x_k)>u^*-\frac{1}{k}, \quad  \text{ii) } \, q(x_k)Hu(x_k)(v,v)<\frac{1}{k}|v|^2, \quad \text{and } \,
|\nabla u(x_k)|<\frac{1}{k}
\eeq
for each  $k\in\mathbb{N}$ and every $v\in T_{x_k}M, v\neq 0$.
\end{thmB'}

In section \ref{nonlinearcase} below we generalize Theorems A and B to a large class of non-linear operators
containing, for instance, the $p$-Laplacian, with $p>1$, the mean curvature operator and so on. Of course Theorems
A' and B' admit similar generalizations to the non-linear case for $\mathcal{C}^2$-solutions. We leave the interested
reader to state the results and provide her/his own proofs following arguments similar to those of Theorems A'' and B''.

\section{Proof of Theorem A and related results}
In this section we give the proof of Theorems A and of some companion results.
\begin{proof}[Proof of Theorem A]
We fix $\eta>0$ and let
\beq\label{aeta}
A_{\eta}=\left\{x\in M:\ u(x)>u^*-\eta\right\}.
\eeq
We claim that
\beq\label{cla1}
\inf_{A_{\eta}}\left\{q(x)Lu(x)\right\}\leq0.
\eeq
Note that (\ref{cla1}) is equivalent to conclusion (\ref{woy}) of Theorem A.

We reason by contradiction and we suppose that
\beq\label{contr1}
q(x)Lu(x)\geq\sigma_0>0 \quad \hbox{on $A_{\eta}$}.
\eeq
First we observe that $u^*$ cannot be attained at any point $x_0\in M$, for otherwise $x_0\in A_{\eta}$, $\nabla u(x_0)=0$, and
$Lu(x_0)$ reduces to $Lu(x_0)=\tr(T\circ\Hess(u))(x_0)$, so that, since $T$ is positive semi-definite, $q(x_0)Lu(x_0)\leq0$ contradicting (\ref{contr1}).

Next we let
\beq
\Omega_t=\left\{x\in M:\ \gamma(x)>t\right\},
\eeq
and define
\beq
u^*_t=\sup_{x\in\Omega^c_t}u(x).
\eeq
Clearly $\Omega^{c}_t$ is closed; we show that it is also compact. In fact, by (\ref{Gam}) $i)$ there exists a compact
set $K_t$ such that $\gamma(x)>t$ for every $x\notin K_t$. In other words, $\Omega^c_t\subset K_t$ and hence it is also
compact. In particular, $u^*_t=\max_{x\in\Omega^c_t}u(x)$.

Since $u^*$ is not attained in $M$ and $\left\{\Omega^c_t\right\}$ is a nested family exhausting $M$, we find a divergent
sequence $\left\{t_j\right\}\subset\mathbb{R}_0^+$ such that
\beq\label{w(t)}
u^*_{t_j}\rightarrow u^* \quad \hbox{as $j\rightarrow+\infty$},
\eeq
and we can choose $T_1>0$ sufficiently large in such a way that
\beq
\label{2.7}
u^*_{T_1}>u^*-\frac{\eta}{2}.
\eeq
Furthermore we can suppose to have chosen $T_1$ sufficiently large that $q(x)>0$ and
(\ref{Gam}) $ii)$ holds on $\Omega_{T_1}$. We choose $\alpha$ such that $u_{T_1}^*<\alpha<u^*$. Because of (\ref{w(t)}) we can find $j$ sufficiently large that
\beq
T_2=t_j>T_1 \quad \hbox{and} \quad u^*_{T_2}>\alpha.
\eeq
We select $\overline{\eta}>0$ small enough that
\beq\label{alphaeta}
\alpha+\overline{\eta}<u^*_{T_2}.
\eeq

For $\sigma\in(0,\sigma_0)$ we define
\beq\label{gammasigma}
\gamma_{\sigma}(x)=\alpha+\sigma(\gamma-T_1).
\eeq
We note that
\beq
\gamma_{\sigma}(x)=\alpha \quad \hbox{for every $x\in\partial\Omega_{T_1}$},
\eeq
and
\beq
\label{luis3}
q(x)L\gamma_{\sigma}(x)=\sigma q(x)L\gamma(x)\leq\sigma B<\sigma_0 \quad \hbox{on $\Omega_{T_1}$},
\eeq
up to have chosen $\sigma$ sufficiently small.

Since on $\Omega_{T_1}\setminus\Omega_{T_2}$ we have
\beq
\alpha\leq\gamma_{\sigma}(x)\leq\alpha+\sigma(T_2-T_1)
\eeq
we can choose $\sigma\in(0,\sigma_0)$ sufficiently small, so that
\beq
\label{luis1}
\sigma(T_2-T_1)<\overline{\eta}
\eeq
and then
\beq
\alpha\leq\gamma_{\sigma}(x)<\alpha+\overline{\eta} \quad \hbox{on $\quad \Omega_{T_1}\setminus\Omega_{T_2}$}.
\eeq
For any such $\sigma$, we have on $\partial\Omega_{T_1}$
\beq
\gamma_{\sigma}(x)=\alpha>u^*_{T_1}\geq u(x),
\eeq
so that
\beq
\label{luis2}
(u-\gamma_{\sigma})(x)<0 \quad \hbox{on $\partial\Omega_{T_1}$}.
\eeq
Furthermore, if $\overline{x}\in\Omega_{T_1}\setminus\Omega_{T_2}$ is such that
\[
u(\overline{x})=u^*_{T_2}>\alpha+\overline{\eta}
\]
then
\[
(u-\gamma_{\sigma})(\overline{x})
\geq u^*_{T_2}-\alpha-\sigma(T_2-T_1)>u^*_{T_2}-\alpha-\overline{\eta}>0
\]
by (\ref{alphaeta}) and \rf{luis1}. Finally, (\ref{Gam}) $i)$ and the fact that $u^*<+\infty$ imply
\beq
\label{2.18}
(u-\gamma_{\sigma})(x)<0 \quad \hbox{on $\Omega_{T_3}$}
\eeq
for $T_3>T_2$ sufficiently large. Therefore,
\[
m=\sup_{x\in\overline{\Omega}_{T_1}}(u-\gamma_\sigma)(x)>0,
\]
and it is in fact a
positive maximum attained at a certain point $z_0$ in the compact set
$\overline{\Omega}_{T_1}\setminus\Omega_{T_3}$. In particular, $\nabla (u-\gamma_\sigma)(z_0)=0$ and
$L(u-\gamma_\sigma)(z_0)$ reduces to $\tr(T\circ\Hess(u-\gamma_\sigma))(z_0)$. Therefore, since $T$ is positive
semi-definite we have that $Lu(z_0)\leq L\gamma_\sigma(z_0)$.

By \rf{luis2} we know that  $\gamma(z_0)>T_1$. Therefore, at $z_0$ we have
\beq
u(z_0)=\gamma_{\sigma}(z_0)+m>\gamma_{\sigma}(z_0)>\alpha>u^*_{T_1}>u^*-\frac{\eta}{2},
\eeq
and hence $z_0\in A_{\eta}\cap\Omega_{T_1}$. In particular $q(z_0)>0$ and (\ref{Gam}) $ii)$ holds at $z_0$.
From (\ref{contr1}) we have
\beq
0<\sigma_0\leq q(z_0)Lu(z_0)\leq q(z_0)L\gamma_{\sigma}(z_0)\leq\sigma B<\sigma_0,
\eeq
which is a contradiction.
\end{proof}

We observe that we can relax the assumption in Theorem A on the boundedness of the function $u$ from above to a control
of $u$ at infinity via the function $\gamma$. This is the content of the next result.
\begin{thmAbis}
\label{thmAbis}
Let $\M$ be a Riemannian manifold and $L=L_{T,X}$ be as above. Let $q(x)\in\mathcal{C}^0(M)$, $q(x)\geq0$ and suppose that
\beq
\label{e1.6b}
q(x)>0 \ \hbox{outside a compact set.}
\eeq
Let $\gamma\in\mathcal{C}^2(M)$ be such that
\beq
\tag{$\Gamma$}
\begin{cases}
\text{i)} \quad \gamma(x)\rightarrow +\infty & \hbox{as $x\rightarrow\infty$,} \\
\text{ii)} \quad q(x)L\gamma(x)\leq B & \hbox{outside a compact set}
\end{cases}
\eeq
for some constant $B>0$. If $\uin$ and
\beq
\label{eqAbis}
u(x)=o(\gamma(x)) \ \hbox{as $x\rightarrow\infty$,}
\eeq
then for each $\mu$ such that
\[
A_\mu=\{x\in M : u(x)>\mu \}\neq\emptyset
\]
we have
\[
\inf_{A\mu}\{ q(x)Lu(x)\}\leq 0.
\]
\end{thmAbis}
\begin{proof}
Of course we consider here the case $u^*=+\infty$. We reason by contradiction as in the proof of Theorem A
and we suppose the validity of \rf{contr1} on $A_\mu$. Proceed as in the above proof (obviously in this case
$u^*$ is not attained on $M$) to arrive to \rf{w(t)} that now takes the form
\beq
\label{2.6bis}
u^*_{t_j}\rightarrow+\infty \ \hbox{as $j\rightarrow\infty$,}
\eeq
and choose $T_1>0$ sufficiently large in such a way that \rf{2.7} becomes now
\beq
\label{2.7bis}
u^*_{T_1}>2\mu.
\eeq
Furthermore we can suppose to have chosen $T_1$ sufficiently large that $q(x)>0$ and
(\ref{Gam}) $ii)$ holds on $\Omega_{T_1}$. We choose $\alpha$ such that $\alpha>u_{T_1}^*$. Because of
(\ref{2.6bis}) we can find $j$ sufficiently large that
\beq
\label{2.8bis}
T_2=t_j>T_1 \quad \hbox{and} \quad u^*_{T_2}>\alpha.
\eeq
Proceed now up to \rf{2.18} which is now true on $\Omega_{T_3}$ for $T_3$ sufficiently large since
\[
(u-\gamma_\sigma)(x)=\gamma_\sigma(\frac{u}{\gamma_\sigma}-1)(x),
\]
expression which becomes negative on $\Omega_{T_3}$, for $T_3$ sufficiently large, because of condition \rf{eqAbis}.

The rest of the proof is as in that of Theorem A.
\end{proof}

We now show the validity of Remark \ref{rmkA}. Thus we assume (\ref{Gam}) $ii)'$ with $G$ as in (\ref{G1}). We
set
\beq
\varphi(t)=\int^t_{t_0}\frac{ds}{G(s)+A\,s\log s}
\eeq
on $[t_0,+\infty)$ for some $t_0>0$. Note that, by (\ref{G1}) $i)$, $\varphi(t)\rightarrow+\infty$ as $t\rightarrow+\infty$. Thus, defining
$\widehat{\gamma}=\varphi(\gamma)$, (\ref{Gam}) $i)$ implies that
\beq\label{sigma1}
\widehat{\gamma}(x)\rightarrow+\infty \quad \hbox{as $x\rightarrow\infty$}.
\eeq
Next, using that
\[
L(\varphi(u))=\varphi'(u)Lu+\varphi''(u)T(\nabla u,\nabla u),
\]
a computation gives
\begin{eqnarray*}
q(x)L\widehat{\gamma}(x) & = & \frac{q(x)L\gamma(x)}{G(\gamma(x))+A\gamma(x)\log\gamma(x)}\\
{} & {} & -\frac{G^{\prime}(\gamma(x))+A(1+\log\gamma(x))}
{\left(G(\gamma(x))+A\gamma(x)\log\gamma(x)\right)^2}q(x)T(\nabla\gamma(x),\nabla\gamma(x))
\end{eqnarray*}
outside a sufficiently large compact set. Since $T(\nabla\gamma,\nabla\gamma)\geq0$, $q(x)\geq0$ and (\ref{G1}) $ii)$ holds, we deduce
\beq
q(x)L\widehat{\gamma}(x)\leq\frac{q(x)L\gamma(x)}{G(\gamma(x))+A\gamma(x)\log\gamma(x)}
\eeq
if $\gamma(x)$ is sufficiently large. Thus, from (\ref{Gam'}) $ii)'$ and $G\geq0$ we finally obtain
\beq\label{sigma2}
q(x)L\widehat{\gamma}(x)\leq B
\eeq
outside a compact set. Then (\ref{sigma1}) and (\ref{sigma2}) show the validity of (\ref{Gam}) $i)$, $ii)$ for
the function $\widehat{\gamma}$.

This finishes the proof of Remark \ref{rmkA}. Regarding Theorem \^{A}, if we substitute \rf{Gam} ii) with
\rf{Gam} $ii)'$, $G$ satisfying \rf{G1}, then condition \rf{eqAbis} has to be substituted by
\beq
\label{eqAbis'}
u(x)=o\left(\int_0^{\gamma(x)}\frac{ds}{G(s)+As\log{s}}\right) \ \hbox{as $x\rightarrow\infty$}.
\eeq
Thus for instance if $G(t)=t$, so that we can choose $A=0$, \rf{Gam} ii)' is $q(x)L\gamma(x)\leq\gamma(x)$
but \rf{eqAbis'} becomes $u(x)=o(\log\gamma(x))$ as $x\rightarrow\infty$, showing a balancing effect between
the two conditions.

\begin{proof}[Proof of Theorem A']
For the proof of Theorem A' we proceed as in the proof of Theorem A letting
\beq
\label{aeta'}
A_\eta=\{x\in M : u(x)>u^*-\eta \}.
\eeq
We claim that for every $\varepsilon>0$ there exists $x\in A_\eta$ such that
\[
q(x)Hu(x)(v,v)<\varepsilon
\]
for each $v\in T_xM$ with $|v|=1$. By contradiction, suppose that there exists $\sigma_0>0$ such that, for every $x\in A_\eta$ there
exists $\bar{v}\in T_xM$, $|\bar{v}|=1$, such that
\beq
\label{contr1'}
q(x)Hu(x)(\bar{v},\bar{v})\geq\sigma_0.
\eeq
Now we follow the argument of the proof of Theorem A up to equation \rf{luis3}, which is now substituted by
\beq
\label{luis3'}
q(x)H\gamma_\sigma(x)(\bar{v},\bar{v})=\sigma q(x)H\gamma(x)(\bar{v},\bar{v})\leq\sigma B<\sigma_0 \quad \text{on } \Omega_{T_1},
\eeq
up to have chosen $\sigma$ sufficiently small. We then proceed up to the existence of a certain point $z_0$ in the compact
set $\overline{\Omega}_{T_1}\setminus\Omega_{T_3}$ where the function $u-\gamma_\sigma$ attains its positive maximum.
In particular, $\nabla (u-\gamma_\sigma)(z_0)=0$ and $H(u-\gamma_\sigma)(z_0)$ reduces to
\[
H(u-\gamma_\sigma)(z_0)(v,v)=T(\mathrm{hess}(u-\gamma_\sigma)(z_0)v,v) \quad \text{ for every } v\in T_{z_0}M.
\]
Therefore, since $T$ is positive semi-definite we have
\[
Hu(z_0)(v,v)\leq H\gamma_\sigma(z_0)(v,v)
\]
for every $v\in T_{z_0}M$.

As in the proof of Theorem A, we have that $z_0\in A_{\eta}\cap\Omega_{T_1}$. In particular $q(z_0)>0$ and (\ref{Gam'}) $ii)'$
holds at $z_0$. On the other hand, from (\ref{contr1'}) we have
\beq
0<\sigma_0\leq q(z_0)Hu(z_0)(\bar{v},\bar{v})\leq q(z_0)H\gamma_{\sigma}(z_0)(\bar{v},\bar{v})\leq\sigma B<\sigma_0,
\eeq
which is a contradiction.
\end{proof}

\section{Proof of Theorem B and some related results}
We follow the notation of the previous section to give the proof of Theorem B.
\begin{proof}[Proof of Theorem B]
We first observe that, although it is not required in the statement, the two assumptions \rf{GamB} $i)$ and $iii)$ imply
that the manifold $M$ is geodesically complete. To see this, let $\varsigma:[0,\ell)\rightarrow M$ be any divergent
path parametrized by arc-length. Here by \textit{divergent} path we mean a path that eventually lies outside any
compact subset of $M$. From \rf{GamB} $iii)$  we have that $|\nabla\gamma|\leq A$ outside a compact subset $K$ of $M$. We set $h(t)=\gamma(\varsigma(t))$ on $[t_0,\ell)$, where $t_0$ has been
chosen so that $\varsigma(t)\notin K$ for all $t_0\leq t<\ell$. Then, for every $t\in[t_0,\ell)$ we have
\[
|h(t)-h(t_0)|=\left|\int_{t_0}^{t}h'(s)ds\right|
\leq\int_{t_0}^{t}|\nabla\gamma(\varsigma(s))|ds\leq A(t-t_0).
\]
Since $\varsigma$ is divergent, then $\varsigma(t)\rightarrow\infty$ as $t\rightarrow\ell^{-}$, so that
$h(t)\rightarrow+\infty$ as $t\rightarrow\ell^{-}$ because of assumption \rf{GamB} $i)$. Therefore, letting
$t\rightarrow\ell^{-}$ in the inequality above, we conclude that $\ell=+\infty$. This shows that divergent paths
in $M$ have infinite length. In other words, the metric on $M$ is complete.

As in the proof of Theorem A we fix $\eta>0$ but, instead of the set $A_{\eta}$ of (\ref{aeta}), we now consider
the set
\beq
B_{\eta}=\left\{x\in M:\ u(x)>u^*-\eta\ \hbox{and}\ |\nabla u(x)|<\eta\right\}.
\eeq
Since the manifold is complete, by applying Ekeland quasi-minimum principle (see for instance \cite{Ek}) we
deduce that $B_{\eta}\neq\emptyset$. We claim that
\beq\label{cla2}
\inf_{B_{\eta}}\left\{q(x)Lu(x)\right\}\leq0.
\eeq
Note that (\ref{cla2}) is equivalent to conclusion (\ref{oy}) of Theorem B. We reason by contradiction and suppose that
\beq
q(x)Lu(x)\geq\sigma_0>0 \quad \hbox{on $B_{\eta}$}.
\eeq
Now the proof follow the pattern of that of Theorem A with the choice of $T_1$, such that also (\ref{Gam}) $iii)$ holds
on $\Omega_{T_1}$. We observe that in this case
\beq
\gamma_{\sigma}(x)=\alpha \quad \hbox{for every $x\in\partial\Omega_{T_1}$},
\eeq
\beq
q(x)L\gamma_{\sigma}(x)=\sigma q(x)L\gamma(x)\leq\sigma B<\sigma_0 \quad \hbox{on $\Omega_{T_1}$},
\eeq
and
\beq
|\nabla\gamma_\sigma(x)|=\sigma|\nabla\gamma(x)|\leq\sigma A<\eta \quad \hbox{on $\Omega_{T_1}$},
\eeq
up to have chosen $\sigma$ sufficiently small.

Therefore, we find a point $z_0\in\overline{\Omega}_{T_1}\setminus\Omega_{T_3}$ where $u-\gamma_{\sigma}$ attains a
positive absolute maximum $m$. As in the proof of Theorem A, $z_0\in\Omega_{T_1}$ and at $z_0$ we have
\beq
u(z_0)>\gamma_{\sigma}(z_0)>\alpha>u^*_{T_1}>u^*-\frac{\eta}{2}>u^*-\eta;
\eeq
furthermore
\beq
|\nabla u(z_0)|=|\nabla\gamma_{\sigma}(z_0)|=\sigma|\nabla\gamma(z_0)|\leq\sigma A<\eta,
\eeq
by our choice of $\sigma$. Thus $z_0\in B_{\eta}\cap\Omega_{T_1}$ and a contradiction is achieved as at the end
of the proof of Theorem A.
\end{proof}

We note that the validity of Remark \ref{rmkB} is immediate. Indeed defining $\widehat{\gamma}=\varphi(\gamma)$ as in the
previous subsection, conditions (\ref{GamB}) $i)$, $ii)$ are satisfied for $\widehat{\gamma}$; as for condition
(\ref{GamB}) $iii)$, using (\ref{GamB'}) $iii)'$ and $G\geq0$, we have
\beq
|\nabla\widehat{\gamma}|=\frac{|\nabla\gamma|}{G(\gamma)+A\gamma\log\gamma}\leq\frac{G(\gamma)}{G(\gamma)+A\gamma\log\gamma}\leq 1
\eeq
outside a compact set. Thus, we also have the validity of (\ref{GamB}) $iii)$ for $\widehat{\gamma}$.

\begin{rmk}
%
\label{rmk3.1}
As mentioned in the Introduction, if $\M$ is a complete, non-compact Riemannian manifold then a natural candidate
for $\gamma(x)$ is $r(x)$. However, $r(x)$ is not $\mathcal{C}^2$ in $C=\left\{o\right\}\cup\cut(o)$ and assumptions (\ref{Gam}) $ii)$ (in Theorem A) and
(\ref{GamB}) $ii)$ and $iii)$ (in Theorem B) have to
be understood and assumed in the weak sense. Nevertheless the proof of Theorem A (and that of Theorem B) still works in this
case, adding in both the assumption $1/q\in L^1_{loc}(M)$ (see section \rl{nonlinearcase} for more details). Indeed, the only problem is at the end of the proof if the point $z_0$ where $u-\gamma_{\sigma}$ attains
its positive absolute maximum $m>0$ is in $C$. However, $u-\gamma_{\sigma}$ is now given by
$f=u-\alpha-\sigma(r-T_1)$ and to avoid the problem we use a trick of Calabi as follows \cite{Ca}. Take any point $z$ where the function
$f$  attains its positive absolute maximum. In case $z\notin C$ then
\[
|\nabla u(z)|=\sigma|\nabla r(z)|=\sigma<\eta.
\]
Otherwise, if $z\in C$, let $\varsigma$ be a minimizing geodesic, parametrized by arclenght, joining $o$ to $z$. For $\ep>0$ suitably
small let $o_{\ep}=\varsigma(\ep)$ and $r_{\ep}(x)=\textrm{dist}_{M}(x,o_{\ep})$. Thus $z\notin\cut(o_{\ep})$
and $r_{\ep}(x)$ is smooth around $z$. Consider the function
\beq
f_{\ep}=u-\alpha-\sigma(r_{\ep}+\ep-T_1).
\eeq
Using the triangle inequality we have
\beq
f_{\ep}(x)-f(x)=\sigma(r(x)-r_\ep(x)-\ep)\leq 0
\eeq
in a neighborhood of $z$. But on $\varsigma|_{[\ep,r(z)]}$, $f_{\ep}=f$ since
\[
r(\varsigma(t))=\textrm{dist}_{M}(o,o_{\ep})+\textrm{dist}_{M}(o_{\ep},\varsigma(t))
=r_\ep(x)+\ep.
\]
Therefore $z$ is also a local maximum for $f_{\ep}$ which is $\mathcal{C}^2$ in a neighborhood of $z$.
Thus,
at $z$
\beq
\label{luis4}
|\nabla u(z)|=\sigma|\nabla r_{\ep}(z)|=\sigma<\eta
\eeq
up to have chosen $\sigma$ sufficiently small.

To complete the proof of Theorem A in this case we proceed as follows. We let
\beq
K=\left\{x\in\Omega_{T_1}:\ (u-\gamma_{\sigma})(x)=f(x)=m \right\},
\eeq
where now $\Omega_t=\{ x\in M: r(x)>t\}$. For very $x\in K$ we have
\[
u(x)=\alpha+\sigma(r(x)-T_1)+m>\alpha>u^*-\frac{\eta}{2},
\]
so that $K\subset A_\eta$.
Fix $z_0\in K$ and choose $0<\mu<m$ sufficiently near to $m$ so that
the connected component $\Lambda_{z_0}$ of the set
\beq
\left\{x\in\Omega_{T_1}:\ (u-\gamma_{\sigma})(x)>\mu)\right\}
\eeq
containing $z_0$ is contained in $A_{\eta}$. Note that $\Lambda_{z_0}$ is bounded by \rf{2.18}. From \rf{contr1} and \rf{luis3}, we have
\beq
Lu(x)\geq\frac{\sigma_0}{q(x)}>L\gamma_{\sigma}(x)
\eeq
on $A_{\eta}\cap\Omega_{T_1}$ in the weak sense. Moreover, $u=\gamma_{\sigma}+\mu$ on the boundary of $\Lambda_{z_0}$.
Applying Theorem 5.3 of \cite{PuRS} (the request $v<\delta$ is vacuous in our case) we deduce that
$u\leq\gamma_{\sigma}+\mu$ on $\Lambda_{z_0}$. But $z_0\in\Lambda_{z_0}$ and from the above we have
$m\leq\mu$, contradiction.

As for completing the proof of Theorem B, we follow the same reasoning replacing $A_\eta$ by $B_\eta$. For doing it,
simple observe that $K\subset B_\eta$ by \rf{luis4}.
\end{rmk}

We omit the details of the proof of Theorem B', which follows similarly from the proof of Theorem B.

A typical application of Theorem B is the following "a priori" estimate. Note that condition \rf{M1.63} below coincides
(for $f=F$) with the Keller-Osserman condition for the Laplace-Beltrami operator (see \cite{FPR}) showing that in this type
of results what really matters is the structure, in this case linear, of the differential operator.
\begin{thm}
\label{thmapriori}
Assume on $\M$ the validity of the $q$-maximum principle for the operator $L=L_{T,X}$ and suppose that
\beq
\label{M1.60}
q(x)T(\cdot,\cdot)\leq C\g{\cdot}{\cdot}
\eeq
for some $C>0$. Let $u\in\mathcal{C}^2(M)$ be a solution of the differential inequality
\beq
\label{M1.61}
q(x)Lu\geq\varphi(u,|\nabla u|)
\eeq
with $\varphi(t,y)$ continuous in $t$, $\mathcal{C}^2$ in $y$ and such that
\beq
\label{M1.62}
\frac{\partial^2\varphi}{\partial y^2}(t,y)\geq 0.
\eeq
Set $f(t)=\varphi(t,0)$. Then a sufficient condition to guarantee that
\[
u^*=\sup_Mu<+\infty
\]
is the existence of a continuous function $F$ positive on $[a,+\infty)$ for some $a\in\mathbb{R}$, satisfying
the following
\beq
\label{M1.63}
\left(\int_a^{t}F(s)ds\right)^{-1/2}\in L^1(+\infty),
\eeq
\beq
\label{M1.64}
\limsup_{t\rightarrow+\infty}\frac{\int_a^tF(s)ds}{tF(t)}<+\infty,
\eeq
\beq
\label{M1.65}
\liminf_{t\rightarrow+\infty}\frac{f(t)}{F(t)}>0
\eeq
and
\beq
\label{M1.66}
\liminf_{t\rightarrow+\infty}\frac{\left(\int_a^{t}F(s)ds\right)^{-1/2}}{F(t)}\frac{\partial\varphi}{\partial y}(t,0)>-\infty.
\eeq
Furthermore, in this case, we have
\beq
\label{M1.67}
f(u^*)\leq 0,
\eeq
\end{thm}

\begin{proof}
Following the proof of Theorem 1.31 in \cite{PRS2} we choose $g\in\mathcal{C}^2(\mathbb{R})$ to be increasing
from $1$ to $2$ on $(-\infty,a+1)$ and defined by
\[
g(t)=\int_{a+1}^{t}\frac{ds}{\left(\int_a^{s}F(r)dr\right)^{1/2}}+2 \quad on \quad [a+1,+\infty).
\]
Observe that
\beq
\label{M1.1}
g'(t)=\frac{1}{\left(\int_a^{t}F(s)ds\right)^{1/2}} \quad
and \quad
g''(t)=-\frac{F(t)}{2}g'(t)^3<0
\eeq
on $(a+1,+\infty)$. We reason by contradiction and assume that $u^*=+\infty$. Since $g$ is increasing,
\[
\inf_M\frac{1}{g(u)}=\frac{1}{g(u^*)}=\frac{1}{g(+\infty)}>0.
\]
By applying the $q$-maximum principle for $L$ to $1/g$, there exists a sequence $\{x_k\}\subset M$ such that
\beq
\label{M1.69}
\lim_{k\rightarrow+\infty}\frac{1}{g(u(x_k))}=\frac{1}{g(+\infty)},
\eeq
or equivalently
\beq
\label{M1.69bis}
\lim_{k\rightarrow+\infty}u(x_k)=+\infty,
\eeq
\beq
\label{M1.70}
|\nabla \frac{1}{g(u)}(x_k|=\frac{g'(u(x_k))}{g(u(x_k))^2}|\nabla u(x_k)|<\frac{1}{k}
\eeq
and finally
\begin{eqnarray}
\label{M1.71}
\nonumber -\frac{1}{k}<q(x_k)L(\frac{1}{g(u)})(x_k) & = & q(x_k)\left\{-\frac{g'(u(x_k))}{g(u(x_k))^2}Lu(x_k)+\right.\\
{} & {} & \left.+\left(\frac{2g'(u(x_k))^2}{g(u(x_k))^3}-\frac{g''(u(x_k))}{g(u(x_k))^2}\right)T(\nabla u(x_k),\nabla u(x_k))\right\}
\end{eqnarray}
Because of \rf{M1.69bis}, we can suppose that the sequence $\{x_k\}$ satsifies $u(x_k)>a+1$, so that \rf{M1.1}
holds along the sequence $u(x_k)$. Multiplying \rf{M1.71} by
\[
\frac{g'(u(x_k))^2}{-g(u(x_k))^2g''(u(x_k))}>0
\]
and using \rf{M1.61}, we obtain
\begin{eqnarray}
\label{M1.2}
\nonumber \frac{g'(u(x_k))^3}{g(u(x_k))^4|g''(u(x_k))|}\varphi(u(x_k),|\nabla u(x_k)|)\leq
\frac{1}{k}\frac{g'(u(x_k))^2}{g(u(x_k))^2|g''(u(x_k))|}+\\
{} \\
\nonumber +\left(\frac{2g'(u(x_k))^4}{g(u(x_k))^5|g''(u(x_k))|}+\frac{g'(u(x_k))^2}{g(u(x_k))^4}\right)q(x_k)T(\nabla u(x_k),\nabla u(x_k)).
\end{eqnarray}
Since $g\geq 1$, then $1/g^2\leq 1/g$ and
\[
\frac{g'(u(x_k))^2}{g(u(x_k))^2|g''(u(x_k))|}\leq \frac{g'(u(x_k))^2}{g(u(x_k))|g''(u(x_k))|}.
\]
On the other hand, by \rf{M1.60} we also have
\[
q(x_k)T(\nabla u(x_k),\nabla u(x_k))\leq C|\nabla u(x_k)|^2.
\]
Using these two facts in \rf{M1.2}, jointly with \rf{M1.70}, yields
\begin{eqnarray}
\label{M1.3}
\nonumber \frac{g'(u(x_k))^3}{g(u(x_k))^4|g''(u(x_k))|}\varphi(u(x_k),|\nabla u(x_k)|)\leq
\frac{g'(u(x_k))^2}{g(u(x_k))|g''(u(x_k))|}\left(\frac{1}{k}+\frac{2C}{k^2}\right)+\frac{C}{k^2}.
\end{eqnarray}
Next, we use Taylor formula with respect to $y$ centered at $(u(x_k),0)$ and \rf{M1.62} to have
\[
\varphi(u(x_k),|\nabla u(x_k)|)\geq f(u(x_k))+\frac{\partial\varphi}{\partial y}(u(x_k),0)|\nabla u(x_k)|,
\]
so that
\beq
\label{M1.5}
\frac{g'(u(x_k))^3f(u(x_k))}{g(u(x_k))^4|g''(u(x_k))|}+A_k\leq \frac{g'(u(x_k))^2}{g(u(x_k))|g''(u(x_k))|}\left(\frac{1}{k}+\frac{2C}{k^2}\right)+\frac{C}{k^2},
\eeq
where
\[
A_k:=\min\left\{0,\frac{1}{k}\frac{\partial\varphi}{\partial y}(u(x_k),0)\frac{g'(u(x_k))^2}{g(u(x_k))^2|g''(u(x_k))|}\right\}.
\]
In what follows, we always assume that $t$ is taken sufficiently large. Observe that we have
\[
\frac{g'(t)^2}{g(t)|g''(t)|}=2\frac{(\int_a^tF(s)ds)^{1/2}}{g(t)F(t)}=
2\frac{\int_a^tF(s)ds}{g(t)(\int_a^tF(s)ds)^{1/2}F(t)},
\]
and
\[
g(t)\geq\frac{t-a-1}{(\int_a^tF(s)ds)^{1/2}},
\]
so that
\[
\frac{g'(t)^2}{g(t)|g''(t)|}\leq c\frac{\int_a^tF(s)ds}{tF(t)}, \quad t\gg 1,
\]
for some positive constant $c$. Therefore, using \rf{M1.64} we deduce
\[
\limsup_{k\rightarrow+\infty}\frac{g'(u(x_k))^2}{g(u(x_k))|g''(u(x_k))|}<+\infty,
\]
and then
\beq
\label{limsup1}
\limsup_{k\rightarrow+\infty}\frac{g'(u(x_k))^2}{g(u(x_k))|g''(u(x_k))|}\left(\frac{1}{k}+\frac{2C}{k^2}\right)+\frac{C}{k^2}=0.
\eeq
On the other hand,
\[
\frac{g'(t)^3f(t)}{g(t)^4|g''(t)|}=\frac{2f(t)}{g(t)^4F(t)}\geq c\frac{f(t)}{F(t)}
\]
for some $c>0$, since $\sup_Mg<+\infty$ by \rf{M1.63}. Therefore, using \rf{M1.65} we have
\beq
\label{M1.6}
\liminf_{k\rightarrow+\infty}\frac{g'(u(x_k))^3f(u(x_k))}{g(u(x_k))^4|g''(u(x_k))|}>0
\eeq
Finally, observe that
\[
\frac{\partial\varphi}{\partial y}(t,0)\frac{g'(t)^2}{g(t)^2|g''(t)|}=\frac{1}{g(t)^2}
\left(\frac{\partial\varphi}{\partial y}(t,0)\frac{(\int_a^tF(s)ds)^{1/2}}{F(t)}\right)
\]
whence, using $\sup_Mg<+\infty$ and \rf{M1.66}, we get
\[
\liminf_{t\rightarrow+\infty}\left(\frac{\partial\varphi}{\partial y}(t,0)\frac{g'(t)^2}{g(t)^2|g''(t)|}\right)>-\infty.
\]
Thus,
\beq
\label{M1.7}
\liminf_{k\rightarrow+\infty}A_k=0.
\eeq
Therefore, taking $k\rightarrow+\infty$ in \rf{M1.5} and using \rf{limsup1}, \rf{M1.6} and \rf{M1.7} we obtain the
desired contradiction.

As for the conclusion $f(u^*)\leq 0$, we note that if $\varphi$ were continuous in both variables, then to reach the desired conclusion
it would be enough to apply the $q$-maximum principle to $u$ to get a sequence $\{y_k\}$ with $\lim u(y_k)=u^*$, $\lim|\nabla u(y_k)|=0$ and
\[
\frac{1}{k}>q(y_k)Lu(y_k)\geq\varphi(u(y_k),|\nabla u(y_k)|).
\]
Thus, taking the limit as $k\rightarrow+\infty$ we would get $f(u^*)\leq 0$. On the other hand, in our more general
assumptions, we can argue in the following way. We re-define the
function $g(t)$ at the very beginning of the proof in such a way that it changes concavity only once at the point $T=\min\{u^*,a\}-1$.
We emphasize that with this choice $g''<0$ on $(T,+\infty)$. We now proceed as in the proof of the first part of the Theorem,
applying the $q$-maximum principle to the function $1/g(u)$, and get the existence of a sequence $\{x_k\}$ as before, with
$g''(u(x_k))<0$ if $k$ is sufficiently large. That is all we need to arrive at \rf{M1.5}. Taking the limit in
this latter for $k\rightarrow+\infty$ and using $\lim_{k\rightarrow+\infty}u(x_k)=u^*<+\infty$, we conclude that
$f(u^*)\leq 0$.
\end{proof}

\section{An application to hypersurfaces into non-degenerate Euclidean cones}
\label{hypersurfaces}

We begin with a general observation. Consider $\M$ a complete, non-compact Riemannian manifold, let $o\in M$ be a reference point,
denote by $r(x)$ the Riemannian distance from $o$, and let $D_o=M\setminus\mathrm{cut}(o)$ be the domain of the
normal geodesic coordinates centered at $o$. Assume that
\[
K_{\mathrm{rad}}\geq -G(r)^2,
\]
where $K_{\mathrm{rad}}$ denotes the radial sectional curvature of $M$, and $G\in\mathcal{C}^1(\mathbb{R}^+_0)$
satisfy
\beq
\label{luis5}
\textrm{i)} \,\,\, G(0)>0, \,\,\, \textrm{ii)} \,\,\, G'(t)\geq 0,  \,\,\, \textrm{and iii)} \,\,\,
\frac{1}{G}\notin L^1(+\infty)
\eeq
Using the general Hessian comparison theorem of \cite{PRS3} one has
\beq
\label{hess}
\Hess(r)\leq\frac{g'(r)}{g(r)}(\g{}{}-dr\otimes dr)
\eeq
on $D_o$, where $g(t)$ is the (positive on $\mathbb{R}^+$) solution of the Cauchy
problem
\beq
\label{cau1}
\left\{
\begin{array}{l}
g''(t)-G(t)^2g(t)=0, \\
g(0)=0, \quad g'(0)=1.
\end{array}
\right.
\eeq
Now let
\beq
\label{NL3.1}
\psi(t)=\frac{1}{G(0)}\left(e^{\int_0^tG(s)ds}-1\right).
\eeq
Then $\psi(0)=0$, $\psi'(0)=1$ and
\beq
\psi''(t)-G(t)^2\psi(t)=\frac{1}{G(0)}\left(G(t)^2+G'(t)\,e^{\int_0^tG(s)ds}\right)\geq 0,
\eeq
that is, $\psi$ is a subsolution of \rf{cau1}. By Sturm comparison theorem
\beq
\frac{g'(t)}{g(t)}\leq\frac{\psi'(t)}{\psi(t)}\leq C G(t)
\eeq
where the last inequality holds for a constant $C>0$ and $t$ sufficiently large.
Hence, from  \rf{hess} and for $r$ sufficiently large
\beq
\Hess(r)\leq C G(r)\g{\ }{\ }.
\eeq
Thus, given the symmetric positive semi-definite  $(2,0)$-tensor $T$ we have
\beq
Lr=\tr(T\circ\Hess(r))\leq C(\tr T) \,G(r) \quad \hbox{for $r>>1$}.
\eeq
Assume that $\tr T>0$ (equivalently, $T\neq 0$) outside a compact set of $M$. Then
\beq
\label{revised1}
\frac{1}{\tr T}Lr\leq C G(r)
\eeq
on $D_o$ and for $r$ sufficiently large. Since $|\nabla r|=1$, if $\mathrm{cut}(o)=\emptyset$ request \rf{GamB}
of Theorem B is satisfied; otherwise we have to prove the validity of \rf{revised1} weakly outside a sufficiently large
ball $B_R$. Since
\[
Lu=\tr(T\circ\Hess(u))=\div (T(\nabla u,\ )^\sharp)-\div T(\nabla u),
\]
we have to show that,
for every $\psi\in\mathcal{C}^\infty_0(M\setminus\overline{B}_R)$, $\psi\geq 0$,
\[
-\int_{M\setminus\overline{B}_R}\left(T(\nabla r,\nabla\psi)+\div T(\nabla r)\psi)\right)\leq C
\int_{M\setminus\overline{B}_R}{\tr T}G(r)\psi,
\]
and this can be obtained as in the proof of Lemma 4.1 in \cite{PuRS} under the assumption that
\beq
\label{revised2}
T(\nabla r,\nu)\geq 0 \quad \text{ in } \Omega,
\eeq
for an exhaustion of $M\setminus\mathrm{cut}(o)$ with smooth bounded domains $\Omega$, star-shaped with respect to $o$, where
$\nu$ denotes the outwards normal along $\partial\Omega$. We have thus proved the validity of the following
\begin{pro}\label{trT}
Let $\M$ be a complete, non-compact Riemannian manifold whose radial sectional curvature satisfies
\begin{equation}
\Krad\geq-G(r)^2
\end{equation}
with $G\in\mathcal{C}^1(\mathbb{R}^+)$ as in \rf{luis5}. Let $T$ be a symmetric, positive semi-definite,
$(2,0)$-tensor field such that $T\neq 0$ outside a compact set of $M$.
Assume that either $\mathrm{cut}(o)=\emptyset$ or otherwise that \rf{revised2} holds.
Then the $q$-Omori-Yau maximum principle holds on
$M$ for the operator $L=\tr(T\circ\Hess)$ with $q=1/\tr T$.
\end{pro}

Now we shall apply Proposition \ref{trT} when $T$ is the $k$-th Newton tensor of an isometrically immersed oriented
hypersurface into the Euclidean space for which, from now and till the end of this section, we assume the validity
of $\mathrm{cut}(o)=\emptyset$ or otherwise that of \rf{revised2}. Note that for $T=I$ \rf{revised2} is automatically satisfied.
Thus, let $\varphi:M^m\rightarrow\R^{m+1}$ denote such an immersion of a connected, $m$-dimensional
Riemannian manifold and assume that it is oriented by a globally defined normal unit vector $N$.
Let $A$ denote the second fundamental form of the immersion with respect to $N$. Then, the $k$-mean curvatures of the hypersurface are
given by
$$
H_k= {m \choose k}^{-1}S_k,
$$
where $S_0=1$ and, for $k=1,\ldots,m$, $S_k$ is the $k$-th elementary symmetric function of the principal curvatures of
the
hypersurface. In particular, $H_1=H$ is the mean curvature, $H_m$ is the Gauss-Kronecker curvature and
$H_2$ is, up to a constant, the scalar curvature of $M$.

The Newton tensors $P_k:TM\rightarrow TM$ associated to the immersion are defined inductively by $P_0=I$ and
\[
P_k=S_k I-AP_{k-1},  \quad 1\leq k\leq m.
\]
Note, for further use, that
\[
\mathrm{Tr}P_k=(m-k)S_k=c_kH_k \quad \textrm{and} \quad \mathrm{Tr}AP_k=(k+1)S_{k+1}=c_kH_{k+1},
\]
where
$$
c_k=(m-k){m \choose k}=(k+1){m \choose k\!+\!1}.
$$
Associated to each globally defined Newton tensor $P_k:TM\rightarrow TM$, we may consider the second order
differential operator $L_k:\mathcal{C}^{2}(M)\rightarrow\mathcal{C}(M)$ given by
\[
L_k=\mathrm{Tr}(P_k \circ \Hess)=\div (P_k(\nabla u,\ )^\sharp)-\g{\div P_k}{\nabla u},
\]
where $\mathrm{div}P_k=\mathrm{Tr}\nabla P_k$. In particular, $L_0$ is the Laplace-Beltrami operator $\Delta$. Observe that
$L_k$ is semi-elliptic (respectively, elliptic) if and only if $P_k$ is positive semi-definite (respectively, positive definite).
\begin{rmk}
\label{rmk4.2}
In this respect, it is worth pointing out that the ellipticity of the operator $L_1$ is guaranteed by the assumption $H_2>0$. Indeed,
if this happens the mean curvature does not vanish on $M$, because of the basic inequality $H_1^2\geq H_2$. Therefore,
we can choose the normal unit vector $N$ on $M$ so that $H_1>0$. Furthermore
\[
m^2H_1^2=\sum_{j=1}^m\kappa_j^2+m(m-1)H_2>\kappa_i^2
\]
for every $i=1,\ldots, m$, and then the eigenvalues of $P_1$ satisfy $\mu_{i,1}=mH_1-\kappa_i>0$ for every $i$
(see, for instance, Lemma 3.10 in \cite{elbert}). This shows ellipticity of $L_1$. Regarding the operator $L_j$ when
$j\geq 2$, a natural hypothesis to guarantee ellipticity is the existence an elliptic point in $M$, that is, a point
$x\in M$ at which the second fundamental form $A$ is positive definite (with respect to the appropriate orientation).
In fact, it follows from the proof of \cite[Proposition 3.2]{BC} that if $M$ has an elliptic point and $H_{k+1}\neq 0$
on $M$, then each $L_j$, $1\leq j\leq k$ is elliptic.
\end{rmk}

Fix an origin $o\in\R^{m+1}$ and a unit vector $a\in\mathbb{S}^{m}$. For $\theta\in(0,\pi/2)$, we denote by
$\mathcal{C}=\mathcal{C}_{o,a,\theta}$ the non-degenerate cone with vertex $o$, direction $a$ and width $\theta$, that is,
\[
\mathcal{C}=\mathcal{C}_{o,a,\theta}=\{ p\in\R^{m+1}\backslash\{o\} : \langle\frac{p-o}{|p-o|},{a}\rangle\geq \cos\theta \}.
\]
By non-degenerate we mean that it is strictly smaller than a half-space. We consider here
isometrically immersed hypersurfaces $\varphi:M^m\rightarrow\R^{m+1}$ with images inside a non-degenerate cone of $\R^{m+1}$ and, as an
application of Proposition \ref{trT} and motivated by the results in \cite{MaR}, we provide a lower bound for the width of the cone in
terms of higher order mean curvatures of the hypersurface. Specifically, we obtain the following result.
\begin{thm}
\label{thm3}
Let $\varphi:M^m\rightarrow\R^{m+1}$ be an oriented isometric immersion of a complete non-compact Riemannian manifold $M^m$ whose radial sectional
curvatures satisfy
\[
\Krad\geq-G(r)^2
\]
with $G\in\mathcal{C}^1(\mathbb{R}^+)$ as in \rf{luis5}.
Assume that 
$P_k$ is positive semi-definite and $H_k$ does not vanish on $M$, and the validity
of $\mathrm{cut}(o)=\emptyset$ or otherwise that of \rf{revised2}. If $\varphi(M)$ is contained into a non-degenerate cone
$\mathcal{C}=\mathcal{C}_{o,a,\theta}$ as above with vertex at $o\in\R^{m+1}\backslash\varphi(M)$, then
\beq
\label{sup1}
\sup\left(\frac{|H_{k+1}|}{H_k}\right)\geq A_0\frac{\cos^2\theta}{d(\Pi_a,\varphi(M))},
\eeq
where $A_0=\frac{6\sqrt{3}}{25\sqrt{5}}\approx 0.186$,
$\Pi_a$ denote the hyperplane orthogonal to $a$ passing through $o$ and $d(\Pi_a,\varphi(M))$ is the Euclidean distance between
this hyperplane and $\varphi(M)$.
\end{thm}
\begin{proof}
To proving the theorem we shall follow the ideas and make use of some computations performed in the proof of Theorem 1.4 in \cite{MaR}.
We may assume without loss of generality that the vertex of the cone is the origin $0\in\R^{m+1}$, so that there exists $a\in\mathbb{S}^{m}$
and $0<\theta<\pi/2$ such that
\beq
\label{4.16}
\langle\frac{\varphi(x)}{|\varphi(x)|},{a}\rangle\geq \cos\theta
\eeq
for every $x\in M$. Observe that
\[
d(\Pi_a,\varphi(M))=\inf_{x\in M}\g{\varphi(x)}{a}.
\]
We reason by contradiction and assume that \rf{sup1} does not hold. Therefore, there exists $x_0\in M$ such that
\[
\g{\varphi(x_0)}{a}\sup\left(\frac{|H_{k+1}|}{H_k}\right)<A\cos^2\theta
\]
for a positive constant $A<A_0$. For the ease of notation we set $\alpha=\g{\varphi(x_0)}{a}>0$, let $\beta\in(0,1)$ and define the
function
\[
u(x)=\sqrt{\alpha^2+\beta^2\cos^2\theta|\varphi(x)|^2}-\g{\varphi(x)}{a}
\]
for every $x\in M$. Note that, by construction, $u(x_0)>0$. We claim that
\[
u(x)\leq\alpha
\]
for every $x\in M$. Indeed, an algebraic manipulation shows that this is equivalent to
\[
\g{\varphi(x)}{a}^2+2\alpha\g{\varphi(x)}{a}-\beta^2\cos^2\theta|\varphi(x)|^2\geq 0,
\]
which holds true by \rf{4.16} since
\[
\g{\varphi(x)}{a}^2+2\alpha\g{\varphi(x)}{a}-\beta^2\cos^2\theta|\varphi(x)|^2\geq
\g{\varphi(x)}{a}^2-\cos^2\theta|\varphi(x)|^2\geq 0.
\]

Next, we consider the closed non-empty set
\[
\overline{\Omega}_0=\{x\in M: u(x)\geq u(x_0)\}.
\]
For every $x\in\overline{\Omega}_0$ and using \rf{4.16} one has
\[
\sqrt{\alpha^2+\beta^2\cos^2\theta|\varphi(x)|^2}\geq u(x_0)+\g{\varphi(x)}{a}\geq u(x_0)+\cos\theta|\varphi(x)|>0.
\]
Squaring this inequality yields
\[
(1-\beta^2)\cos^2\theta|\varphi(x)|^2+2u(x_0)\cos\theta|\varphi(x)|+u(x_0)^2-\alpha^2\leq 0
\]
for every $x\in\overline{\Omega}_0$. The left half side of the above inequality is a quadratic polynomial in $|\varphi(x)|$ with two distinct roots
$\alpha_{-}<0<\alpha_{+}$ given by
\[
\alpha_{\pm}=\frac{\pm\sqrt{\beta^2u(x_0)^2+(1-\beta^2)\alpha^2}-u(x_0)}{(1-\beta^2)\cos\theta}.
\]
Therefore, for every $x\in\overline{\Omega}_0$ it holds that
\[
0<|\varphi(x)|\leq\alpha_{+}=\frac{\sqrt{\beta^2u(x_0)^2+(1-\beta^2)\alpha^2}-u(x_0)}{(1-\beta^2)\cos\theta}.
\]
Using the elementary inequality $\sqrt{1+t^2}\leq 1+t$ for $t\geq 0$, we have
\begin{eqnarray*}
\alpha_{+} & = &
\frac{1}{(1-\beta^2)\cos\theta}\left(\sqrt{\beta^2u(x_0)^2\left(1+\frac{(1-\beta^2)\alpha^2}{\beta^2u(x_0)^2}\right)}-u(x_0)\right)\\
{} & = &
\frac{\beta u(x_0)}{(1-\beta^2)\cos\theta}\sqrt{1+\frac{(1-\beta^2)\alpha^2}{\beta^2u(x_0)^2}}-\frac{u(x_0)}{(1-\beta^2)\cos\theta}\\
{} & \leq &
\frac{\beta u(x_0)}{(1-\beta^2)\cos\theta}\left(1+\frac{\sqrt{1-\beta^2}\alpha}{\beta u(x_0)}\right)-\frac{u(x_0)}{(1-\beta^2)\cos\theta}\\
{} & = & \frac{\alpha}{\sqrt{1-\beta^2}\cos\theta}-\frac{u(x_0)}{(1+\beta)\cos\theta}\\
{} & \leq & \frac{\alpha}{\sqrt{1-\beta^2}\cos\theta}.
\end{eqnarray*}
Therefore,
\beq
\label{4.44}
|\varphi(x)|\leq\frac{\alpha}{\sqrt{1-\beta^2}\cos\theta} \quad \textrm{ on } \,\, \overline{\Omega}_0.
\eeq

To compute $L_ku=\tr(P_k\circ\Hess u)$ when $P_k$ is the $k$-th Newton tensor, we first observe that
\beq
\label{4.45}
\nabla u=-a^\top+\frac{\beta^2\cos^2\theta}{\sqrt{\alpha^2+\beta^2\cos^2\theta|\varphi|^2}}\varphi^\top,
\eeq
where, as usual, ${}^\top$ denotes tangential component along the immersion $\varphi$. That is,
\[
a=a^\top+\g{a}{N}N \,\, \textrm{ and } \,\, \varphi=\varphi^\top+\g{\varphi}{N}N.
\]
Using that
\[
\nabla_Xa^\top=\g{a}{N}AX
\]
and
\[
\nabla_X\varphi^\top=X+\g{\varphi}{N}AX
\]
for every $X\in TM$, we get from \rf{4.45} that
\begin{eqnarray}
\label{4.46}
\nonumber \nabla^2u(X,Y)=\g{\nabla_X\nabla u}{Y} & = & \frac{\beta^2\cos^2\theta}{\sqrt{\alpha^2+\beta^2\cos^2\theta|\varphi|^2}}\g{X}{Y}\\
{} & {} &  +\langle{\frac{\beta^2\cos^2\theta}{\sqrt{\alpha^2+\beta^2\cos^2\theta|\varphi|^2}}\varphi-a},{N}\rangle\g{AX}{Y}\\
\nonumber {} & {} & +\frac{-\beta^4\cos^4\theta}{(\alpha^2+\beta^2\cos^2\theta|\varphi|^2)^{3/2}}\g{X}{\varphi^\top}\g{Y}{\varphi^\top},
\end{eqnarray}
for every $X,Y\in TM$. Hence,
\begin{eqnarray}
L_ku=\sum_{i=1}^{m}\nabla^2u(e_i,Pe_i)& = & \langle\frac{\xi}{|\varphi|}\varphi-a,N\rangle\tr(A\circ P_k)+\frac{\xi}{|\varphi|}\tr(P_k)\\
{} & {} & -\frac{\xi^2}{|\varphi|^2}\frac{1}{\sqrt{\alpha^2+\beta^2\cos^2\theta|\varphi|^2}}\g{P_k\varphi^\top}{\varphi^\top},
\end{eqnarray}
where
\[
\xi(x)=\frac{\beta^2\cos^2\theta|\varphi(x)|}{\sqrt{\alpha^2+\beta^2\cos^2\theta|\varphi(x)|^2}}.
\]
That is,
\beq
\label{4.471}
L_ku=c_k\langle\frac{\xi}{|\varphi|}\varphi-a,N\rangle H_{k+1}+c_k\frac{\xi}{|\varphi|}H_k
-\frac{\xi^2}{|\varphi|^2}\frac{1}{\sqrt{\alpha^2+\beta^2\cos^2\theta|\varphi|^2}}\g{P_k\varphi^\top}{\varphi^\top}.
\eeq
Observe that, by \rf{4.16},
\beq
\label{4.475}
\left|\frac{\xi}{|\varphi|}\varphi-a\right|^2=\xi^2-2\xi\frac{\g{\varphi}{a}}{|\varphi|}+1\leq
\xi^2-2\cos\theta\xi+1\leq 1,
\eeq
since $0<\xi(x)<\beta\cos\theta$ for every $x\in M$. On the other hand, since $P_k$ is positive semi-definite we have
\beq
\label{4.47}
0\leq\g{P_k\varphi^\top}{\varphi^\top}\leq\tr(P_k)|\varphi^\top|^2\leq c_kH_k|\varphi|^2.
\eeq
Since, by our hypothesis, $H_k>0$ on $M$, we obtain from here that
\begin{eqnarray}
\label{3.175}
\nonumber \frac{1}{c_kH_k}L_ku & \geq & -\frac{|H_{k+1}|}{H_k}+\frac{\xi}{|\varphi|}-\frac{\xi^2}{\sqrt{\alpha^2+\beta^2\cos^2\theta|\varphi|^2}}\\
{} & \geq & -\sup\frac{|H_{k+1}|}{H_k}+\frac{\alpha^2\beta^2\cos^2\theta}{(\alpha^2+\beta^2\cos^2\theta|\varphi|^2)^{3/2}}
\end{eqnarray}
on $M$. Recall that, from our choice of $x_0$, we have
\[
\sup\frac{|H_{k+1}|}{H_k}<A\frac{\cos^2\theta}{\alpha}
\]
for a positive constant $A<A_0=\frac{6\sqrt{3}}{25\sqrt{5}}$. On the other hand, by \rf{4.44} we also have
\beq
\label{4.495}
|\varphi|^2<\frac{\alpha^2}{(1-\beta^2)\cos^2\theta}
\eeq
on $\overline{\Omega}_0$. This yields
\[
\frac{\alpha^2\beta^2\cos^2\theta}{(\alpha^2+\beta^2\cos^2\theta|\varphi|^2)^{3/2}}\geq\frac{\cos^2\theta}{\alpha}\beta^2(1-\beta^2)^{3/2}
\]
on $\overline{\Omega}_0$. Choose $\beta=\sqrt{2/5}$. Then, $\beta^2(1-\beta^2)^{3/2}=A_0$ and
\beq
\label{4.51}
\frac{1}{c_kH_k}L_ku\geq\frac{\cos^2\theta}{\alpha}\left(A_0-A\right)>0 \quad \textrm{ on } \overline{\Omega}_0.
\eeq
There are now two possibilities:
\begin{itemize}
\item[i)] $x_0$ is an absolute maximum for $u$ on $M$. Then, $L_ku(x_0)\leq 0$, contradicting \rf{4.51}.
\item[ii)] $\Omega_0=\{ x\in M : u(x)>u(x_0)\}\neq\emptyset$. In this case, since $u(x)$ is bounded above on $M$ it is enough
to evaluate inequality \rf{4.51} along a sequence $\{x_k\}$ realizing the $1/c_kH_k$-weak maximum principle for the operator $L_k$ on $M$. This
latter holds because of Proposition \rl{trT} and the assumptions of the theorem. We thus have $u(x_k)>u^*-1/k$ and therefore
$x_k\in\Omega_0$ for $k$ sufficiently large and
\[
0<\frac{\cos^2\theta}{\alpha}\left(A_0-A\right)\leq\frac{1}{c_kH_k}L_ku(x_k)<\frac{1}{k}.
\]
By taking $\lim_{k\rightarrow\infty}$ in this inequality we get a contradiction.
\end{itemize}
This completes the proof of the theorem.
\end{proof}

\begin{cor}
\label{coro1}
Let $\varphi:M^m\rightarrow\R^{m+1}$ be an oriented isometric immersion of a complete non-compact Riemannian manifold $M^m$
whose radial sectional curvatures satisfy
\[
\Krad\geq-G(r)^2
\]
with $G\in\mathcal{C}^1(\mathbb{R}^+)$ as in \rf{luis5}.
Assume that $P_k$ is positive semi-definite, and the validity
of $\mathrm{cut}(o)=\emptyset$ or otherwise that of \rf{revised2}. If $\varphi(M)$ is contained into a non-degenerate cone
$\mathcal{C}=\mathcal{C}_{o,a,\theta}$ as above with vertex at $o\in\R^{m+1}\backslash\varphi(M)$, then
\beq
\label{ineqcoro1}
\sup|H_{k+1}|\geq A_0\frac{\cos^2\theta}{d(\Pi_a,\varphi(M))}\inf H_k,
\eeq
where $A_0=\frac{6\sqrt{3}}{25\sqrt{5}}\approx 0.186$,
$\Pi_a$ denote the hyperplane orthogonal to $a$ passing through $o$ and $d(\Pi_a,\varphi(M))$ is the Euclidean distance between
this hyperplane and $\varphi(M)$.
\end{cor}

For the proof of Corollary \ref{coro1} observe that
(\ref{ineqcoro1}) holds trivially if $\inf_M H_k=0$. If $\inf_M H_k>0$, then $H_k>0$ everywhere and the result follows directly
from Theorem \ref{thm3} since the estimate (\ref{ineqcoro1}) is
weaker than (\ref{sup1}).

On the other hand, in the case of $k=1$ we can slightly improve our Theorem \ref{thm3}, both regarding the condition on the ellipticity of
$P_1$ and the value of the constant $A_0$ in \rf{sup1}. Specifically we prove the following.
\begin{cor}
Let $\varphi:M^m\rightarrow\R^{m+1}$ be an oriented isometric immersion of a complete non-compact Riemannian manifold $M^m$ whose radial sectional
curvatures satisfy
\[
\Krad\geq-G(r)^2
\]
with $G\in\mathcal{C}^1(\mathbb{R}^+)$ as in \rf{luis5}. Assume the validity
of $\mathrm{cut}(o)=\emptyset$ or otherwise that of \rf{revised2}.
If $H_2>0$ (equivalently, the scalar curvature of $M$ is positive) and $\varphi(M)$ is contained into a non-degenerate cone
$\mathcal{C}=\mathcal{C}_{o,a,\theta}$ as above with vertex at $o\in\R^{m+1}\backslash\varphi(M)$, then
\beq
\label{sup1b}
\sup\sqrt{H_2}\geq\sup\left(\frac{H_2}{H_1}\right)\geq B_m\frac{\cos^2\theta}{d(\Pi_a,\varphi(M))},
\eeq
where $B_2=B_3=A_0=\frac{6\sqrt{3}}{25\sqrt{5}}\approx 0.186$, and
$$
B_m=\max_{0<\varrho<1}\left(\varrho^2\sqrt{1-\varrho^2}(1-\frac{3}{m}\varrho^2)\right)
$$
for $m\geq 4$.
\end{cor}
We emphasize that $B_m>A_0$ and $B_m\sim 2/(3\sqrt{3})\approx 0.385$ when $m$ goes to infinity.
\begin{proof}
According to Remark \rl{rmk4.2}, the assumption $H_2>0$ and $m^2H_1^2-|A|^2=m(m-1)H_2>0$ guarantee that $P_1$ is positive definite
for an appropriate choice of the unit normal $N$, so that $H_1>0$ and $mH_1-|A|>0$ on $M$.

By Cauchy-Schwarz inequality,
\[
H_1^2-H_2=\frac{1}{m(m-1)}\left(\sum_{i=1}^m\kappa^2_i-\frac{1}{m}\left(\sum_{i=1}^m\kappa_i\right)^2\right)\geq 0.
\]
This immediately yields $H_2/H_1\leq\sqrt{H_2}$ and gives the first inequality in \rf{sup1b}.

As for the second inequality in \rf{sup1b}, arguing as in the proof of Theorem \ref{thm3}, we reason by contradiction
and assume that there exists a point $x_0\in M$ such that
\beq
\label{4.75a}
\alpha\sup\left(\frac{H_{2}}{H_1}\right)<A\cos^2\theta
\eeq
for a positive constant $A<B_m$, where $\alpha=\g{\varphi(x_0)}{a}$. We then follow the proof of Theorem \ref{thm3} until we reach \rf{4.471}, which jointly with \rf{4.475} yields
\[
L_1u\geq -c_1H_{2}+c_1\frac{\xi}{|\varphi|}H_1-\frac{\xi^2}{|\varphi|^2}
\frac{1}{\sqrt{\alpha^2+\beta^2\cos^2\theta|\varphi|^2}}\g{P_1\varphi^\top}{\varphi^\top}.
\]
The idea to improve the value of the constant $A_0$ in \rf{sup1} is to improve the estimate \rf{4.47} in the following way. Using
that $P_1=mH_1I-A$ we have
\beq
\label{3.17b}
\g{P_1\varphi^\top}{\varphi^\top}=mH_1|\varphi^\top|^2-\g{A\varphi^\top}{\varphi^\top}\leq 2mH_1|\varphi|^2,
\eeq
because of the fact that
\[
|\g{A\varphi^\top}{\varphi^\top}|\leq|A||\varphi^\top|^2\leq mH_1|\varphi|^2.
\]
Note that \rf{3.17b} gives a better estimate than \rf{4.47} for $k=1$ when $m\geq 4$. In that case, making use of \rf{3.17b} we
obtain
\begin{eqnarray*}
\frac{1}{c_1H_1}L_1u & \geq & -\frac{H_{2}}{H_1}+\frac{\xi}{|\varphi|}-
\frac{2}{m-1}\frac{\xi^2}{\sqrt{\alpha^2+\beta^2\cos^2\theta|\varphi|^2}}\\
{} & \geq & -\sup\frac{H_{2}}{H_1}+
\frac{\alpha^2\beta^2\cos^2\theta+\frac{m-3}{m}\beta^4\cos^4\theta|\varphi|^2}{(\alpha^2+\beta^2\cos^2\theta|\varphi|^2)^{3/2}}
\end{eqnarray*}
on $M$, instead of \rf{3.175}. It follows from \rf{4.495} that
\[
\frac{\alpha^2\beta^2\cos^2\theta+\frac{m-3}{m}\beta^4\cos^4\theta|\varphi|^2}{(\alpha^2+\beta^2\cos^2\theta|\varphi|^2)^{3/2}}
\geq\frac{\cos^2\theta}{\alpha}\beta^2\sqrt{1-\beta^2}(1-\frac{3}{m}\beta^2)
\]
on $\overline{\Omega}_0$. Choose $\beta\in(0,1)$ to maximize $\varrho^2\sqrt{1-\varrho^2}(1-\frac{3}{m}\varrho^2)$. That is,
\[
\beta^2=\frac{4+m-\sqrt{(4+m)^2-40m/3}}{10}
\]
and
$$
B_m=\beta^2\sqrt{1-\beta^2}(1-\frac{3}{m}\beta^2).
$$
Then,
\beq
\label{4.51b}
\frac{1}{c_1H_1}L_1u\geq\frac{\cos^2\theta}{\alpha}\left(B_m-A\right)>0 \quad \textrm{ on } \overline{\Omega}_0.
\eeq
The proof then finishes as in Theorem \ref{thm3}.
\end{proof}

For the case $k\geq 2$ there is an inequality corresponding to the first one in \rf{sup1b}, given by
\[
\sup_M \sqrt[k+1]{H_{k+1}}\geq\sup_M\left(\frac{H_{k+1}}{H_k}\right).
\]
However, to guarantee its validity ones needs to assume the existence of an elliptic point (see \cite{ADR} for details).

\section{An application to PDE's}
We give a typical application of Theorem A to PDE's in the following comparison theorem. Towards this end let us
introduce the next definition: A function $f:\mathbb{R}^+\rightarrow\mathbb{R}^+$ is said to be $\zeta$-increasing
if for every $\zeta>1$ and for every closed interval $I\subset\mathbb{R}^+$ there exists $A=A(\zeta,I)>0$ such that
\beq
\label{F0}
\frac{f(\zeta t)}{f(t)}\geq 1+A
\eeq
for every $t\in I$. Note that this implies that $tf(t)$ is strictly increasing on $\mathbb{R}^+$. Typical examples of $\zeta$-increasing
functions are $f(t)=t^\sigma\log^{a}(1+t)$ with $\sigma\geq 1, a\geq 0$, $f(t)=t^\sigma e^{at}$ with $\sigma\geq 0, a>0$, and
so on.
\begin{thm}
\label{Fthm}
Let $a(x), b(x)\in\mathcal{C}^0(M)$ and $f\in\mathcal{C}^1(\mathbb{R}^+)$ be a $\zeta$-increasing function.
Assume that
\beq
\label{F1}
i) \,\, b(x)>0 \quad on \quad M \quad and \quad ii) \, \sup_M\frac{a_{-}}{b}<+\infty,
\eeq
where, as usual, $a_{-}$ denotes the negative part of $a$. For $L=L_{T,X}$ as in our previous notation, let
$u,v\in\mathcal{C}^2(M)$ be non-negative solutions of
\beq
\label{F2}
Lu+a(x)u-b(x)uf(u)\geq 0\geq Lv+a(x)v-b(x)vf(v)
\eeq
on $M$ satisfying
\beq
\label{F3}
i) \, v(x)\geq C_1, \quad ii) \, u(x)\leq {C_2}
\eeq
outside some compact set $K\subset M$ for some positive constants $C_1,C_2$. Then $$u(x)\leq v(x)$$ on $M$ provided
that the $1/b$-weak maximum principle holds for $L$.
\end{thm}
As an immediate consequence, we have
\begin{cor}
In the assumption of Theorem \ref{Fthm}, the equation
\[
Lu+a(x)u-b(x)uf(u)=0
\]
has at most one non-negative, non-trivial, bounded solution $u$ with $\liminf_{x\rightarrow\infty}u(x)>0$.
\end{cor}

\begin{proof}[Proof of Theorem \ref{Fthm}]
We can assume that $u\not\equiv 0$, otherwise, there is nothing to prove. Next, the differential inequality
\[
Lv+a(x)v-b(x)vf(v)\leq 0
\]
and \rf{F3} i), together with the strong maximum principle (see the observation after the proof of Theorem 3.5 at
page 35 of \cite{GT}), imply $v>0$ on $M$. This fact and \rf{F3} tells us that
\beq
\label{F4}
\zeta=\sup_M\frac{u}{v}
\eeq
satisfies
\[
0<\zeta<+\infty.
\]
If $\zeta\leq 1$ then $u\leq v$ on $M$. Let us assume by contradiction that $\zeta>1$ and define
\[
\varphi=u-\zeta v
\]
Note that $\varphi\leq 0$ on $M$ and it is not hard to realize, using \rf{F3} and \rf{F4}, that
\beq
\label{F5}
\sup_M\varphi=0.
\eeq
We now use \rf{F2} and the linearity of $L$ to compute
\beq
\label{F6}
L\varphi\geq -a(x)\varphi+b(x)\left[uf(u)-\zeta vf(\zeta v)\right]+b(x)\zeta v\left[f(\zeta v)-f(v)\right].
\eeq
Let
\[
h(x)=\left\{\begin{array}{cc}
\left[f(u)+uf'(u)\right](x) & \mathrm{if} \quad u(x)=\zeta v(x) \\
{} & \\
\displaystyle{\frac{1}{u(x)-\zeta v(x)}\int_{\zeta v(x)}^{u(x)}\left[f(t)+tf'(t)\right]dt} & \mathrm{if} \quad u(x)<\zeta v(x).
\end{array} \right.
\]
Observe that $h$ is continuous on $M$ and non-negative, since
\[
(tf(t))'=f(t)+tf'(t)\geq 0 \quad on \quad \mathbb{R}^{+}.
\]
Furthermore, we can re-write \rf{F6} in the form
\[
L\varphi\geq \left[-a(x)+b(x)h(x)\right]\varphi+b(x)\zeta v\left[f(\zeta v)-f(v)\right],
\]
and using $-a(x)\varphi\geq a_{-}(x)\varphi$ we get
\beq
\label{F7}
L\varphi\geq \left[a_{-}(x)+b(x)h(x)\right]\varphi+b(x)\zeta v\left[f(\zeta v)-f(v)\right],
\eeq
Let
\[
\Omega_{-1}=\{x\in M : \varphi(x)>-1 \}.
\]
On $\Omega_{-1}$ we have
\beq
\label{F8}
v(x)=\frac{1}{\zeta}(u(x)-\varphi(x))\leq\frac{1}{\zeta}(C+1)
\eeq
for some positive constant $C$, since $u$ is bounded above on $M$. Using definition of $h$ and the mean value theorem
for integrals, we deduce
\[
h(x)=f(y)+yf'(y)
\]
for some $y=y(x)\in [u(x),\zeta v(x)]$. Since $u(x)$ and $v(x)$ are bounded above on $\Omega_{-1}$
\beq
\label{F9}
h(x)\leq C
\eeq
on $\Omega_{-1}$ for some constant $C>0$.

Next we recall that $b(x)>0$ on $M$ to re-write \rf{F7} in the form
\[
\frac{1}{b(x)}L\varphi\geq \left[\frac{a_{-}(x)}{b(x)}+h(x)\right]\varphi+\zeta v\left[f(\zeta v)-f(v)\right].
\]
Since $\varphi\leq 0$, \rf{F1} ii) and \rf{F9} imply
\[
\left[\frac{a_{-}(x)}{b(x)}+h(x)\right]\varphi\geq C\varphi
\]
for some appropriate constant $C>0$ on $\Omega_{-1}$. Thus
\[
\frac{1}{b(x)}L\varphi\geq C\varphi+\zeta v\left[f(\zeta v)-f(v)\right]
\]
on $\Omega_{-1}$. Since $f$ is $\zeta$-increasing, there exists $A>0$ such that
\[
\zeta v\left[f(\zeta v)-f(v)\right]\geq \zeta Avf(v) \quad on \quad \Omega_{-1}.
\]
Now we use the fact that $v$, and hence $vf(v)$, is bounded from below by a positive constant to get
\[
\frac{1}{b(x)}L\varphi\geq C\varphi+B \quad on \Omega_{-1},
\]
for some positive constant $B$. Finally, we choose $0<\varepsilon<1$ sufficiently small such that
\[
C\varphi>-\frac{1}{2}B
\]
on
\[
\Omega_{-\varepsilon}=\{x\in M : \varphi(x)>-\varepsilon \}\subset\Omega_{-1}.
\]
Therefore,
\[
\frac{1}{b(x)}L\varphi\geq \frac{1}{2}B>0 \quad on \quad \Omega_{-\varepsilon}.
\]
Having assume the validity of the $1/b$-weak maximum principle for the operator $L$ on $M$, we immediately get a contradiction, proving
that $\zeta\leq 1$.
\end{proof}

\section{A glimpse at the non-linear case}
\label{nonlinearcase}

In this section we will introduce an extension of Theorems A and B to the non-linear case. Since solutions of PDE's involving the type of operators
we shall consider are not, in general, even for constant coefficients, of class $\mathcal{C}^2$, it will be more appropriate to work, from the very beginning, in the weak setting.
Think for instance of the $p$-Laplace operator with $p\neq 2$, $p>1$.

We let $A:\R^{+}\fle\R{}$ and we define $\varphi(t)=tA(t)$. The next assumptions will be crucial to apply Theorems 5.3 and 5.6 of \cite{PuRS} and shall therefore be assumed all over this section:

\begin{itemize}
\item[(A1)] $A\in\mathcal{C}^1(\R^{+})$.\\

\item[(A2)] i) $\varphi'(t)>0$ on $\R^{+}$,  ii) $\varphi(t)\fle 0$ as $t\fle 0^{+}$.\\

\item[(A3)] $\varphi(t)\leq Ct^{\delta}$ on $(0,\omega)$ for some $\omega,C,\delta>0$.\\

\item[(T1)] $T$ is a positive definite, symmetric, 2-covariant tensor field on $M$.\\

\item[(T2)] For every $x\in M$ and for every $\xi\in T_xM$, $\xi\neq 0$, the bilinear form
\[
\frac{A'(|\xi|)}{|\xi|}\g{\xi}{\cdot}\odot T(\xi,\cdot)+A(|\xi|)T(\cdot,\cdot)
\]
is symmetric and positive definite. Here $\odot$ denotes the symmetric tensor product
\end{itemize}
Note that the above requirements are not mutually independent. Indeed the bilinear form in (T2) is automatically
symmetric when $T$ does. Furthermore, if write it in terms of $\varphi$, that is, for every $x\in M$ and for every
$\xi,v\in T_xM$, $\xi,v\neq 0$,
\[
\frac{1}{|\xi|^2}\left(\varphi'(|\xi|)-\frac{\varphi(|\xi|)}{|\xi|}\right)\g{\xi}{v}T(\xi,v)+\frac{\varphi(|\xi|)}{|\xi|}T(v,v)>0.
\]
In particular, the choice $v=\xi$ shows that
\[
\varphi'(t)>0 \quad \text{on} \quad \R^{+},
\]
that is, requirement i) in (A2). Request (T2) is in fact equivalent to i) in (A2) in case $T=t(x)\g{}{}$ is a
''pointwise conformal'' deformation of the metric for some smooth function $t(x)>0$ on $M$. Indeed, in this case (T2)
reduces to
\[
\frac{1}{|\xi|^2}\varphi'(|\xi|)t(x)\g{\xi}{v}^2+
\frac{\varphi(|\xi|)}{|\xi|^3}t(x)\left(|v|^2|\xi|^2-\g{\xi}{v}^2\right)>0
\]
for every $x\in M$ and for every $\xi,v\in T_xM$, $\xi,v\neq 0$.

Having fixed a vector field $X$ on $M$, we define the following operator $L=L_{A,T,X}$ acting on $\mathcal{C}^1(M)$:
\[
Lu=\div\left(A(|\nabla u|)T(\nabla u,\cdot)^{\sharp}\right)-\g{X}{\nabla u}
\]
for each $u\in\mathcal{C}^1(M)$,
where ${}^\sharp:T^*M\fle TM$ denotes the musical isomorphism.
Of course, the above operator $L$ has to be understood in the appropriate weak sense.

Observe that sometimes we shall refer to $\omega, C$ and $\delta$ in (A3) as to the \textit{structure constants} of the operator $L$.

$L$ gives rise to various familiar operators. For instance, choosing $T=\g{}{}$ and $X=0$ we have
\begin{itemize}
\item[1.] For $\varphi(t)=t^{p-1}$, $p>1$,
\[
Lu=\div\left(|\nabla u|^{p-2}\nabla u\right)
\]
is the usual $p$-Laplacian. Note that for the structural constants we have $C=1$, $\delta=p-1$ and $\omega=+\infty$.
Of course the case $p=2$ yields the usual Laplace-Beltrami operator.
\item[2.] For $\varphi(t)={t}/{\sqrt{1+t^2}}$ the operator
\[
Lu=\div\left(\frac{\nabla u}{\sqrt{1+|\nabla u|^2}}\right)
\]
is the usual mean curvature operator. Here $C=1$, $\delta=1$ and $\omega=+\infty$.
\end{itemize}
And so on.

We let, as in the linear case, $q(x)\in\mathcal{C}^0(M)$, $q(x)\geq 0$, be such that, for some compact $K\subset M$,
$q(x)>0$ on $M\setminus K$. However, since our setting now is that of solutions in the weak sense, for technical reasons
(see for instance \rf{NL1.3} in the proof of Theorem A'' below) we need the local integrability of $1/q$ also inside $K$.
Thus, from now on we suppose
\beq
\label{Q}
\tag{Q}
\frac{1}{q}\in L^1_{loc}(M).
\eeq
This fact was also pointed out in Remark \rl{rmk3.1} of the linear case whenever we deal with functions $u$ on $M$ which are merely of class
$\mathcal{C}^1$.

Next, we introduce the following Khas'minski{\u\i} type condition.
\begin{defi}
\label{q-SK}
We say that the (q-SK) condition holds if there exists a telescoping exhaustion of relatively compact open sets
$\{\Sigma_j\}_{j\in\mathbb{N}}$ such that $K\subset\Sigma_1$, $\overline{\Sigma}_j\subset\Sigma_{j+1}$ for every $j$
and, for any pair $\Omega_1=\Sigma_{j_1}$, $\Omega_2=\Sigma_{j_2}$, with $j_1<j_2$, and for each $\varepsilon>0$, there exists $\gamma\in\mathcal{C}^0(M\setminus\Omega_1)\cap\mathcal{C}^1(M\setminus\overline{\Omega}_1)$ with
the following properties:
\begin{itemize}
\item[i)] $\gamma\equiv 0$ on $\partial\Omega_1$,
\item[ii)] $\gamma>0$ on $M\setminus\Omega_1$,
\item[iii)] $\gamma\leq\varepsilon$ on $\Omega_2\setminus\Omega_1$,
\item[iv)] $\gamma(x)\fle+\infty$ when $x\fle\infty$,
\item[v)] $q(x)L\gamma\leq\varepsilon$ on $M\setminus\overline{\Omega}_1$.
\end{itemize}
\end{defi}
Since property v) has to be intended in the weak sense we mean that
\[
L\gamma\leq\frac{\varepsilon}{q(x)} \text{ weakly on } M\setminus\overline{\Omega}_1,
\]
that is, for all $\psi\in\mathcal{C}^\infty_0(M\setminus\overline{\Omega}_1)$, $\psi\geq 0$,
\[
\int_{M\setminus\overline{\Omega}_1}\left(A(|\nabla\gamma|)T(\nabla\gamma,\nabla\psi)+\g{X}{\nabla\gamma}\psi+\frac{\varepsilon}{q}\psi\right)\geq 0.
\]
Of course we expect the (q-SK) condition in Definition \rl{q-SK} to be equivalent in the linear case to the weak form of
\rf{Gam} of Theorem A, which obviously reads as follows:
\begin{defi}
\label{q-KL}
We say that the (q-KL) condition holds if there exist a compact set $H\supset K$ and a function
$\tilde{\gamma}\in\mathcal{C}^1(M)$ with the following properties:
\begin{itemize}
\item[j)] $\tilde{\gamma}(x)\fle+\infty$ when $x\fle\infty$,
\item[jj)] $q(x)L\tilde{\gamma}\leq B$ on $M\setminus H$ for some constant $B$, in the weak sense.
\end{itemize}
\end{defi}
Obviously, the (q-SK) condition implies the (q-KL) condition simply by choosing $H=\overline{\Omega}_2$, setting
$\tilde{\gamma}=\gamma$ on $M\setminus{\Omega}_2$ and extending it on $\Omega_2$ to be of class $\mathcal{C}^1$ on $M$.
We shall prove the equivalence of the two conditions in the linear case after the proof of Theorem A''.
The point is that in the form (q-SK) the Khas'minski{\u\i} type condition is not only sufficient for the validity of the
$q$-weak maximum principle but indeed equivalent in many cases (see \cite{MV}). For a certain class of operators this happens also in the non-linear case as shown
in \cite{AMR} (in preparation).

Before stating Theorem A'' we recall that for an operator $L$, a function $q(x)>0$ on an open set
$\Omega\subset M$ and $u\in\mathcal{C}^1(\Omega)$
the inequality
\[
\inf_\Omega\{ q(x)Lu(x)\}\leq 0
\]
holds in the weak sense if for each $\varepsilon>0$
\[
-\int_{\Omega}\left(A(|\nabla u|)T(\nabla u,\nabla\psi)+\g{X}{\nabla u}\psi\right)\leq
\int_\Omega\frac{\varepsilon}{q}\psi
\]
for each $\psi\in\mathcal{C}^\infty_0(\Omega)$, $\psi\geq 0$.

We are now ready to state the non-linear version of Theorem A.
\begin{thmA''}
\label{NLthA}
Let $(M,\g{}{})$ be a Riemannian manifold and let $L$ be as above. Let $q(x)\in\mathcal{C}^0(M)$, $q(x)\geq 0$, and
suppose that $q(x)>0$ outside some compact set $K\subset M$ and $1/q\in L^1_{loc}(M)$.
Assume the validity of (q-SK). If $u\in\mathcal{C}^1(M)$ and $u^*=\sup_Mu<+\infty$ then for each $\eta>0$
\beq
\label{NL1.1}
\inf_{A_\eta}\{q(x)Lu(x)\}\leq 0
\eeq
holds in the weak sense, where
\beq
\label{NL1.2}
A_\eta=\{ x\in M : u(x)>u^*-\eta \}.
\eeq
\end{thmA''}
\begin{proof}
We argue by contradiction and we suppose that for some $\eta>0$ there exists $\varepsilon_0>0$ such that
\[
Lu\geq\frac{\varepsilon_0}{q(x)}
\]
holds weakly on $A_\eta$, that is, for each $\psi\in\mathcal{C}^\infty_0(A_\eta)$, $\psi\geq 0$,
\beq
\label{NL1.3}
\int_{A_\eta}\left(A(|\nabla u|)T(\nabla u,\nabla\psi)+\g{X}{\nabla u}\psi+\frac{\varepsilon_0}{q}\psi\right)\leq 0.
\eeq
Note that since in general $A_\eta\not\subset M\setminus K$ it is here essential assumption \rf{Q}.

First we observe that $u^*$ cannot be attained at any point $x_0\in M$. Otherwise $x_0\in A_\eta$ and, because of \rf{NL1.3}, on the open set $A_\eta$ it holds weakly
\beq
\label{NL1.4}
Lu\geq 0.
\eeq
Since, in our assumptions, the strong maximum principle given in Theorem 5.6 of \cite{PuRS} holds, we deduce that $u\equiv u^*$ on the connected component of $A_\eta$ containing $x_0$, which contradicts \rf{NL1.3}.

Next we let $\Sigma_j$ be the telescoping sequence of relatively compact open domains of condition (q-SK). Given $u^*-\frac{\eta}{2}$, there exists $\Sigma_{j_1}$ such that
\[
u^*_{j_1}=\sup_{\overline{\Sigma}_{j_1}}u>u^*-\frac{\eta}{2}.
\]
We set $\Omega_1=\Sigma_{j_1}$ and define
\[
u^*_1=u^*_{j_1}.
\]
Note that, since $u^*$ is not attained on $M$
\beq
\label{NL1.5}
u^*-\frac{\eta}{2}<u^*_1<u^*.
\eeq
We can therefore fix $\alpha$ so that
\beq
\label{NL1.6}
u^*_1<\alpha<u^*.
\eeq
Since $\alpha>u^*_1$, there exists $\Sigma_{j_2}$ with $j_2>j_1$ such that, setting $\Omega_2=\Sigma_{j_2}$, $u^*_2=\sup_{\Omega_2}u=\max_{\bar{\Omega}_2}u$, we have
\[
\overline{\Omega}_1\subset\Omega_2
\]
and furthermore
\beq
\label{NL1.7}
u^*_1<\alpha<u^*_2<u^*.
\eeq

We fix $\bar{\eta}>0$ so small that
\beq
\label{NL1.8}
\alpha+\bar{\eta}<u^*_2
\eeq
and
\beq
\label{NL1.9}
\bar{\eta}<\varepsilon_0.
\eeq
We apply the (q-SK) condition with the choice $\varepsilon=\bar{\eta}$ and $\Omega_1$ and $\Omega_2$ as above to obtain the existence of $\gamma\in\mathcal{C}^0(M\setminus\Omega_1)\cap\mathcal{C}^1(M\setminus\overline{\Omega}_1)$
satisfying the properties listed in Definition \rl{q-SK}. Construct
\beq
\label{NL1.10}
\sigma(x)=\alpha+\gamma(x).
\eeq
Then
\beq
\label{NL1.11}
\sigma(x)=\alpha \text{ on } \partial\Omega_1,
\eeq
\beq
\label{NL1.12}
\alpha<\sigma(x)\leq\alpha+\bar{\eta} \text{ on } \Omega_2\setminus\bar{\Omega}_1,
\eeq
\beq
\label{NL1.13}
\sigma(x)\rightarrow +\infty \text{ as } x\rightarrow\infty,
\eeq
and, since $\nabla\sigma=\nabla\gamma$, $L\sigma=L\gamma$ and by v) of Definition \rl{q-SK}
\beq
\label{NL1.14}
q(x)L\sigma\leq\bar{\eta} \text{ in the weak sense on } M\setminus\bar{\Omega}_1.
\eeq

Next, we consider the function $u-\sigma$. Because  of \rf{NL1.11} and \rf{NL1.6}, we have for every $x\in\partial\Omega_1$
\beq
\label{NL1.15}
(u-\sigma)(x)=u(x)-\alpha\leq u^*_1-\alpha<0.
\eeq
Since $u^*_2=\max_{\bar{\Omega}_2}u$ and $\bar{\Omega}_2$ is compact, $u^*_2$ is attained at some $\bar{x}\in\bar{\Omega}_2$. Note that $\bar{x}\notin\bar{\Omega}_1$ because otherwise
\[
u^*_1\geq u(\bar{x})=u^*_2,
\]
contradicting \rf{NL1.7}. Thus $\bar{x}\in\bar{\Omega}_2\setminus\bar{\Omega}_1$. By \rf{NL1.8} we have
\[
u(\bar{x})>\alpha+\bar{\eta}.
\]
Thus, by \rf{NL1.12} and \rf{NL1.8}, we have
\beq
\label{NL1.16}
(u-\sigma)(\bar{x})=u^*_2-\sigma(\bar{x})\geq u^*_2-\alpha-\bar{\eta}>0.
\eeq
Finally, because of \rf{NL1.13}, there exists $\Sigma_\ell$, $\ell>j_2$, such that
\beq
\label{NL1.17}
(u-\sigma)(x)<0 \text{ on } M\setminus\Sigma_\ell.
\eeq

Because of \rf{NL1.15}, \rf{NL1.16} and \rf{NL1.17} the function $u-\sigma$ attains an absolute maximum $m>0$ at a certain point
$z_0\in\Sigma_\ell\setminus\bar{\Omega}_1\subset M\setminus\bar{\Omega}_1$. At $z_0$, and by \rf{NL1.6} and \rf{NL1.5}, we have
\[
u(z_0)=\sigma(z_0)+m>\sigma(z_0)=\alpha+\gamma(z_0)\geq\alpha>u_1^*>u^{*}-\frac{\eta}{2},
\]
an hence $z_0\in A_\eta$. It follows that
\beq
\label{NL1.18}
\Xi=\{ x\in M\setminus\bar{\Omega}_1 : (u-\sigma)(x)=m \}\subset A_\eta.
\eeq
Since
$A_\eta$ is open there exists a neighborhood $U_\Xi$ of $\Xi$ contained in $A_\eta$. Pick any $y\in\Xi$, fix $\beta\in(0,m)$ and
call $\Xi_{\beta,y}$ the connected component of the set
\[
\{ x\in M\setminus\bar{\Omega}_1 : (u-\sigma)(x)>\beta \}
\]
containing $y$. Since $\beta>0$,
\[
\Xi_{\beta,y}\subset\bar{\Sigma}_\ell\setminus\bar{\Omega}_1\subset M\setminus\bar{\Omega}_1,
\]
and we can also choose $\beta$ sufficiently near to $m$ so that $\bar{\Xi}_{\beta,y}\subset A_\eta$.
Furthermore, $\bar{\Xi}_{\beta,y}$ is compact.
Because of \rf{NL1.14}, \rf{NL1.9} and
\rf{NL1.3}, on $\Xi_{\beta,y}$ we have
\[
q(x)Lu(x)\geq\varepsilon_0>q(x)L\gamma(x)
\]
in the weak sense. Furthermore,
\[
u(x)=\sigma(x)+\beta<\sigma(x) \quad \text{ on } \partial\Xi_{\beta,y}.
\]
Hence by Theorem 5.3 of \cite{PuRS},
\[
u(x)\leq \sigma(x) \quad \text{ on } \Xi_{\beta,y}.
\]
This contradicts the fact that $y\in\Xi_{\beta,y}$, indeed,
\[
u(y)=\sigma(y)+m>\sigma(y)
\]
since $m>0$. This completes the proof of Theorem A''.
\end{proof}

Suppose now that $L$ is linear, that is, $A(t)=1$ (and hence $\varphi(t)=t$). Once (T1) is satisfied, assumptions (A1), (A2), (A3) and (T2)
are also satisfied. Let $q(x)\in\mathcal{C}^0(M)$, $q(x)\geq 0$, be such that, for some compact $K\subset M$,
$q(x)>0$ on $M\setminus K$ and $1/q\in L^1_{loc}(M)$. Observe that in this case the (q-KL) condition and the linearity of
$L$ imply the (q-SK) condition. Indeed, fix a strictly increasing divergent
sequence $\{T_j\}\nearrow+\infty$ and let
\[
\Sigma_j=\{ x\in M : \tilde{\gamma}(x)<T_j \}.
\]
Obviously, each $\Sigma_j$ is open and because of j) in (q-KL) condition one immediately verifies that
$\bar{\Sigma}_j=\{ x\in M : \tilde{\gamma}(x)\leq T_j \}$ is compact. For the same reason we can suppose to have chosen $T_1$ sufficiently large
that $K\subset H\subset \Sigma_1$. Furthermore $\bar{\Sigma}_j\subset\Sigma_{j+1}$ and again by j) in (q-KL) condition
$\{\Sigma_j\}$ is a telescoping exhaustion. Consider any pair
\[
\Omega_1=\Sigma_{j_1}=\{ x\in M : \tilde{\gamma}(x)<T_{j_1} \}
\]
and
\[
\Omega_2=\Sigma_{j_2}=\{ x\in M : \tilde{\gamma}(x)<T_{j_2} \}
\]
with $j_2>j_1$, and choose $\varepsilon>0$. Let $\sigma\in(0,\sigma_0)$ and define $\gamma:M\setminus\Omega_1\rightarrow\mathbb{R}^{+}_0$
by setting
\[
\gamma(x)=\sigma(\tilde{\gamma}(x)-T_{j_1}).
\]
Then
\begin{itemize}
\item[i)] $\gamma(x)=0$ for every $x\in\partial\Omega_1$,
\item[ii)] $\gamma(x)>0$ if $x\in M\setminus\overline{\Omega}_1=\{x\in M : \tilde{\gamma}(x)>T_{j_1}\}$,
\item[iii)] on $\Omega_2\setminus\Omega_1=\{ x\in M : T_{j_1}\leq\tilde{\gamma}(x)<T_{j_2}\}$ we have
$\gamma(x)<\sigma(T_{j_2}-T_{j_1})$ and hence, up to have chosen $\sigma_0$ sufficiently small, $\gamma(x)\leq\varepsilon$ on
$\Omega_2\setminus\Omega_1$,
\item[iv)] $\gamma(x)\fle+\infty$ when $x\fle\infty$, because of j), and
\item[v)] on $M\setminus\overline{\Omega}_1$ and by linearity of $L$,
\[
q(x)L\gamma=q(x)L(\sigma(\tilde{\gamma}-T_{j_1}))=q(x)\sigma L\tilde{\gamma}\leq\sigma B\leq\varepsilon
\]
because of jj) and up to have chosen $\sigma_0$ sufficiently small.
\end{itemize}

\begin{rmk}
It is worth giving some examples where the (q-SK) condition is satisfied. For the sake of simplicity we limit ourselves to the case $T=\g{}{}$ and
$X\equiv 0$. Let $(M,\g{}{})$ be a complete, non-compact Riemannian manifold of dimension $m\geq 2$. Let $o\in M$ be
a fixed reference point, denote by $r(x)$ the Riemannian distance from $o$ and suppose that
\beq
\label{NL2.1}
\Ric(\nabla r,\nabla r)\geq -(m-1)G(r)^2
\eeq
for some positive non-decreasing function $G(r)\in\mathcal{C}^0(\mathbb{R}^{+}_0)$, $G(r)>0$, with $1/G\not\in L^1(+\infty)$. Similarly to what has been made in section
\rl{hypersurfaces} and for the same $\psi$ defined there (see \rf{NL3.1}), by the Laplacian comparison theorem we have
\beq
\label{NL2.2}
\Delta r\leq (m-1)\frac{\psi'}{\psi}(r)
\eeq
weakly on $M$
for $r\geq R_0>0$ sufficiently large.

Suppose now that the function $q(x)\in\mathcal{C}^0(M)$, $q(x)\geq 0$, satisfies
\beq
\label{NL2.4}
q(x)\leq \Theta(r(x))
\eeq
outside a compact set $K\subset M$, for some non-increasing continuous function
$\Theta:\mathbb{R}^{+}_0\rightarrow\mathbb{R}^{+}$ with the property that
\beq
\label{revised6.1}
\Theta(t)\leq B G^{\delta-1}(t)
\eeq
for $t\gg 1$ and some constant $B>0$ (here $\delta$ is as in (A3)). Note that if $\delta\geq 1$, \rf{revised6.1}
is automatically satisfied.

Fix $\sigma>0$ and $R\geq R_0$ such that $K\subset B_{R}$, $B_{R}$ being the geodesic ball of radius $R$,  and define the function
\beq
\label{NL2.5}
\chi_\sigma(r)=\int_{R}^{r}\varphi^{-1}\left(\sigma h(t)\right)dt, \quad r\in [R,+\infty),
\eeq
where
\[
h(t)=\psi^{1-m}(t)\int_{R}^t\frac{\psi^{m-1}(s)}{\Theta(s)}ds.
\]
Since $\varphi:\mathbb{R}^{+}_0\rightarrow[0,\varphi(+\infty))=I\subseteq\mathbb{R}^{+}_0$ increasingly,
$\varphi:I\rightarrow\mathbb{R}^{+}_0$. Therefore in order that $\chi_\sigma$ be well defined when
$\varphi(+\infty)<+\infty$, we need that for every $t\in[R,+\infty)$
\beq
\label{NL2.6}
\sigma h(t)\in I.
\eeq
Towards this end we note that
\beq
\label{NL2.7}
\frac{\psi'}{\psi}(t)=G(t)\frac{e^{\int_0^tG(s)ds}}{e^{\int_0^tG(s)ds}-1}\sim CG(t) \quad \text{ as } \quad
t\rightarrow+\infty.
\eeq
Then
\beq
h(t)\leq\frac{1}{\Theta(t)}\psi^{1-m}(t)\int_{R}^t\psi^{m-1}(s)ds\leq\frac{C}{\Theta(t)G(t)}
\eeq
for $t\gg 1$ and some $C>0$. The assumption
\[
\limsup_{r\rightarrow+\infty}\frac{1}{\Theta(r)G(r)}<+\infty
\]
is therefore enough to guarantee that $h(t)$ is bounded above. By choosing $\sigma$ sufficiently small, say
$0<\sigma\leq\sigma_0$, we obtain the validity of \rf{NL2.6} so that \rf{NL2.5} is well defined on $[R,+\infty)$.

Define $\gamma(x)=\chi_\sigma(r(x))$ for $x\in M\setminus B_{R}$ and note that
\begin{itemize}
\item[i)] $\gamma\equiv 0$ on $\partial B_{R}$,
\item[ii)] $\gamma>0$ on $M\setminus \overline{B_{R}}$,
\end{itemize}
Moreover, having fixed $\varepsilon>0$ and a second geodesic ball $B_{\hat{R}}$ with $\hat{R}>R$, since
$\varphi^{-1}(t)\rightarrow 0$ as $t\rightarrow 0^+$, up to choosing $\sigma>0$ sufficiently small we also have
$\chi_\sigma(r)\leq \varepsilon$ if $R\leq r<\hat{R}$, so that
\begin{itemize}
\item[iii)] $\gamma\leq\varepsilon$ on $B_{\hat{R}}\setminus B_{R}$,
\end{itemize}
On the other hand, since $1/G\not\in L^1(+\infty)$, to prove that
\begin{itemize}
\item[iv)] $\gamma(x)\fle+\infty$ when $x\fle\infty$
\end{itemize}
it suffices to show that
\[
\varphi^{-1}(\sigma h(t))\geq\frac{\hat{C}}{G(t)} \quad \text{ for } t\gg 1
\]
for some constant $\hat{C}>0$. Equivalently, that there exists a constant $\hat{C}>0$ such that
\beq
\label{revised6.2}
\frac{h(t)}{\varphi\left(\frac{\hat{C}}{G(t)}\right)}\geq\frac{1}{\sigma}  \quad \text{ for } t\gg 1.
\eeq
Without loss of generality we can suppose $G(t)\rightarrow+\infty$ as $t\rightarrow+\infty$.
By the structural condition (A3) on $\varphi$ we have
\[
\varphi\left(\frac{\hat{C}}{G(t)}\right)\leq C\frac{\hat{C}^\delta}{G(t)^\delta},
\]
so that
\[
\frac{h(t)}{\varphi\left(\frac{\hat{C}}{G(t)}\right)}\geq\frac{A(t)}{B(t)}
\]
with
\[
A(t)=G(t)^\delta\int_{R}^t\frac{\psi^{m-1}(s)}{\Theta(s)}ds
\]
and
\[
B(t)=C\hat{C}^\delta\psi^{m-1}(t).
\]
Note that both $A(t)$ and $B(t)$ diverge to $+\infty$ as $t\rightarrow+\infty$. Hence,
\[
\liminf_{t\rightarrow+\infty}\frac{A(t)}{B(t)}\geq\liminf_{t\rightarrow+\infty}\frac{A'(t)}{B'(t)}.
\]
A computation that uses $G'\geq 0$, $\Theta>0$ and \rf{revised6.1} shows that
\[
\frac{A'(t)}{B'(t)}\geq\frac{G(t)}{BC\hat{C}^\delta(m-1)\frac{\psi'(t)}{\psi(t)}}, \quad t\gg 1,
\]
and since ${\psi'(t)}/{\psi(t)}\sim G(t)$ as $t\rightarrow+\infty$, we can choose $\hat{C}>0$ sufficiently small that
\[
\liminf_{t\rightarrow+\infty}\frac{A'(t)}{B'(t)}\geq\frac{1}{\sigma},
\]
proving the validity of \rf{revised6.2}

Clearly, by definition, $\chi_\sigma(t)$ is non-decreasing and satisfies $\chi'_\sigma(t)=\varphi^{-1}(\sigma h(t))$,
that is, $\varphi(\chi'_\sigma(t))=\sigma h(t)$. Therefore
\[
\nabla\gamma=\chi'_\sigma(r)\nabla r, \quad |\nabla\gamma|=\chi'_\sigma(r) \quad \text{and} \quad \varphi(|\nabla\gamma|)=\sigma h(r).
\]
Since
\[
h'(t)=\frac{1}{\Theta(t)}-(m-1)\frac{\psi'}{\psi}(t)h(t),
\]
a computation using \rf{NL2.2} and \rf{NL2.4} gives
\begin{eqnarray}
\nonumber L\gamma & = & \div\left(\frac{\varphi(|\nabla\gamma|)}{|\nabla\gamma|}\nabla\gamma\right)=\div(\sigma h(r)\nabla r)
=\sigma h'(r)|\nabla r|^2+\sigma h(r)\Delta r \\
{} & = & \frac{\sigma}{\Theta(r)}+\sigma h(r)\left(\Delta r-(m-1)\frac{\psi'}{\psi}(r)\right)\leq
\frac{\sigma}{\Theta(r)}\leq \frac{\sigma}{q(x)}
\end{eqnarray}
if $r\geq R$. That is,
\begin{itemize}
\item[v)] $q(x)L\gamma\leq\sigma$ on $M\setminus\overline{B_{R}}$
\end{itemize}
outside the cut locus and weakly on all of $M\setminus\overline{B_{R}}$ as it can be easily proved.

It is now clear how to satisfy the requirements of the (q-SK) condition in Definition\rl{q-SK} by choosing a
telescoping exhaustion $\{ B_{R+j}\}_{j\in\mathbb{N}}$.
\end{rmk}

\begin{rmk}
\label{remark6.4}
Here we introduce another example where the (q-SK) condition is satisfied with $T=\g{}{}$ and arbitrary $X$.
Let $(M,\g{}{})$ be a complete, non-compact Riemannian manifold of dimension $m\geq 2$. Let $o\in M$ be
a fixed reference point, denote by $r(x)$ the Riemannian distance from $o$ and suppose, as in the previous example,
that
\beq
\label{NL2.1b}
\Ric(\nabla r,\nabla r)\geq -(m-1)G(r)^2
\eeq
for some positive non-decreasing function $G(r)\in\mathcal{C}^0(\mathbb{R}^{+}_0)$, $G(r)>0$, with
$1/G\not\in L^1(+\infty)$. We know that, for the same function $\psi$,
\beq
\label{NL2.2b}
\Delta r\leq (m-1)\frac{\psi'}{\psi}(r)\leq CG(r)
\eeq
weakly on $M$ for $r\geq R_0>0$ sufficiently large and some $C>0$.

Suppose now that the function $q(x)\in\mathcal{C}^0(M)$, $q(x)\geq 0$, satisfies
\beq
\label{NL2.4b}
q(x)\leq \frac{1}{G(r(x))+|X(x)|}
\eeq
outside a compact set $K\subset M$. Fix $\sigma>0$ and $R\geq R_0$ such that $K\subset B_{R}$, $B_{R}$ being the geodesic
ball of radius $R$ centered at $o$,  and define the function
\beq
\label{NL2.5b}
\gamma(x)=\sigma(r(x)-R) \quad \text{for } x\in M\setminus B_{R}.
\eeq
Obviously,
\begin{itemize}
\item[i)] $\gamma\equiv 0$ on $\partial B_{R}$,
\item[ii)] $\gamma>0$ on $M\setminus \overline{B_{R}}$,
\end{itemize}
Moreover, having fixed $\varepsilon>0$ and a second geodesic ball $B_{\hat{R}}$ with $\hat{R}>R$, up to choosing
$\sigma>0$ sufficiently small we also have
\begin{itemize}
\item[iii)] $\gamma\leq\varepsilon$ on $B_{\hat{R}}\setminus B_{R}$,
\end{itemize}
On the other hand, since $M$ is complete
\begin{itemize}
\item[iv)] $\gamma(x)\fle+\infty$ when $x\fle\infty$
\end{itemize}
Finally, a direct computation using \rf{NL2.2b} and \rf{NL2.4b} gives
\begin{eqnarray}
\nonumber L\gamma & = & \div\left(\frac{\varphi(|\nabla\gamma|)}{|\nabla\gamma|}\nabla\gamma\right)-\g{X}{\nabla\gamma}=
\div(\varphi(\sigma)\nabla r)-\sigma\g{X}{\nabla r}\\
{} & = & \varphi(\sigma)\Delta r -\sigma\g{X}{\nabla r}\leq \varphi(\sigma)CG(r)+\sigma|X|\\
{} & \leq & \varepsilon(G(r)+|X|)\leq \frac{\varepsilon}{q(x)}
\end{eqnarray}
if $r\geq R$, up to choosing $\sigma>0$ sufficiently small, since $\varphi(\sigma)\rightarrow 0$ as $\sigma\rightarrow 0^+$.
That is,
\begin{itemize}
\item[v)] $q(x)L\gamma\leq\varepsilon$ on $M\setminus\overline{B_{R}}$
\end{itemize}
outside the cut locus $\mathrm{cut}(o)$ and weakly on all of $M\setminus\overline{B_{R}}$ as it can be easily proved.
It is now clear how to satisfy the requirements of the (q-SK) condition in Definition\rl{q-SK} by choosing a
telescoping exhaustion $\{ B_{R+j}\}_{j\in\mathbb{N}}$.
\end{rmk}

For the next result we introduce the following strengthening of the (q-SK) condition.
\begin{itemize}
\item[vi)] $|\nabla\gamma|<\varepsilon$ on $M\setminus\Omega_1$.
\end{itemize}

\begin{thmB''}
\label{NLthB}
Let $(M,\g{}{})$ be a Riemannian manifold and let $L$ be as above. Let $q(x)\in\mathcal{C}^0(M)$, $q(x)\geq 0$ satisfy (Q).
Assume the validity of (q-SK$\nabla$). If $u\in\mathcal{C}^1(M)$ and $u^*=\sup_Mu<+\infty$ then for each $\eta>0$
\beq
\label{NL1.19}
\inf_{B_\eta}\{q(x)Lu(x)\}\leq 0
\eeq
holds in the weak sense, where
\[
B_\eta=\{ x\in M : u(x)>u^*-\eta \quad \text{and} \quad |\nabla u(x)|<\eta \}.
\]
\end{thmB''}
\begin{proof}
First of all note that the validity of (q-SK$\nabla$) implies, once we fix arbitrarily a pair $\Omega_1\subset\Omega_2$ ,
an $\varepsilon>0$ and a corresponding $\gamma$, that the metric is geodesically complete. Indeed, let
$\varsigma:[0,\ell)\rightarrow M$ be any divergent
path parametrized by arc-length. Thus $\varsigma$ lies eventually outside any compact subset of $M$. From vi),
$|\nabla\gamma|\leq\varepsilon$ outside the compact subset $\bar{\Omega}_1$. We set $h(t)=\gamma(\varsigma(t))$ on $[t_0,\ell)$,
where $t_0$ has been
chosen so that $\varsigma(t)\notin\bar{\Omega}_1$ for all $t_0\leq t<\ell$. Then, for every $t\in[t_0,\ell)$ we have
\[
|h(t)-h(t_0)|=\left|\int_{t_0}^{t}h'(s)ds\right|
\leq\int_{t_0}^{t}|\nabla\gamma(\varsigma(s))|ds\leq \varepsilon(t-t_0).
\]
Since $\varsigma$ is divergent, then $\varsigma(t)\rightarrow\infty$ as $t\rightarrow\ell^{-}$, so that
$h(t)\rightarrow+\infty$ as $t\rightarrow\ell^{-}$ because of iv). Therefore, letting
$t\rightarrow\ell^{-}$ in the inequality above, we conclude that $\ell=+\infty$. This shows that divergent paths
in $M$ have infinite length and in other words, that the metric is complete.

Since the metric is complete, we can apply Ekeland quasi-minimum principle to deduce that $B_\eta\neq\emptyset$ and therefore that
the infimum in \rf{NL1.19} is meaningful.

Now we proceed as in the proof of Theorem A'' substituting, as in the linear case, the subset $A_\eta$ with the smaller open
set $B_\eta$. We need to show that the compact set $\Xi$ defined in \rf{NL1.18} satisfies
$\Xi\subset B_\eta$. Because of \rf{NL1.8} it is enough to prove that for every $z\in\Xi$,
\beq
\label{NL1.20}
|\nabla u(z)|<\eta.
\eeq
But $z$ is a point of absolute maximum for $(u-\sigma)$ and $z\in M\setminus\bar{\Omega}_1$, hence using vi) of (q-SK$\nabla$),
\[
|\nabla u(z)|=|\nabla \sigma(z)|=|\nabla\gamma(z)|<\varepsilon.
\]
thus $\Xi\subset B_\eta$ and the rest of the proof is now exactly as at the end of Theorem A''. This finishes
the proof of Theorem B''.
\end{proof}

Suppose now that $L$ is linear; we have an analog condition (q-KL), that is, (q-KL$\nabla$), adding
\begin{itemize}
\item[jjj)] $|\nabla\tilde{\gamma}|\leq A$ on $M\setminus H$, for some constant $A>0$.
\end{itemize}
It is immediate to show that this condition and linearity of $L$ imply (q-KS$\nabla$).

\section*{Acknowledgments}
We thank Lucio Mari for helpful discussions on section \rl{nonlinearcase}.

\bibliographystyle{amsplain}

\end{document}